\documentclass[a4paper,10pt]{article}

\title{Overconvergent cohomology, $p$-adic $L$-functions and families for $\mathrm{GL}(2)$ over CM fields}
\author{Daniel Barrera Salazar and Chris Williams}

\usepackage[all]{xy}
\usepackage{tikz-cd}

\date{}

\usepackage{fancyhdr}

\usepackage[margin=1.26in]{geometry}

\usepackage{adjustbox}
\usepackage{dashbox}
\usepackage{wrapfig}
\usepackage[justification=centering]{caption}

\usepackage{amssymb,amsmath,latexsym,amsthm}
\usepackage[english]{babel}
\usepackage{afterpage}

\usepackage{subfiles}

\newcommand{\Addressesshort}{{
		\bigskip
		\footnotesize
		Daniel Barrera Salazar; Universidad de Santiago de Chile $\ \cdot\ $
		\texttt{daniel.barrera.s@usach.cl}
		
		Chris Williams; University of Warwick $ \ \cdot \ $
		\texttt{christopher.d.williams@warwick.ac.uk}\\

}}

\usepackage{preamble}

\usepackage{etoolbox}
\newtoggle{notes}
\toggletrue{notes}
\togglefalse{notes} 
\usepackage{comment}
\usepackage{verbatim}
\usepackage{newclude}

\usepackage[utf8]{inputenc}
\usepackage[T1]{fontenc}
\usepackage{lmodern,graphicx}
\xyoption{pdf}
\xyoption{rotate}

\usepackage[labelfont=bf]{caption}

\newcommand{\cyc}{{\mathrm{cyc}}}

\newcommand{\n}{\mathfrak{n}}

\numberwithin{equation}{section}
\numberwithin{figure}{section}

\begin{document}

\setlength{\abovedisplayskip}{7pt}
\setlength{\belowdisplayskip}{7pt}
\setlength{\abovedisplayshortskip}{7pt}
\setlength{\belowdisplayshortskip}{7pt}

%
%

\pagestyle{fancy}
\maketitle
\renewcommand{\thefootnote}{\fnsymbol{footnote}} 

\footnotetext{\today. \emph{2010 MSC:} Primary 11F41, 11F67, 11F85, 11S40; Secondary 11M41
}

\begin{abstract}
	The use of overconvergent cohomology in constructing $p$-adic $L$-functions, initiated by Stevens and Pollack--Stevens in the setting of classical modular forms, has now been established in a number of settings. The method is compatible with constructions of eigenvarieties by Ash--Stevens, Urban and Hansen, and is thus well-adapted to non-ordinary situations and variation in $p$-adic families. In this note, we give an exposition of the ideas behind the construction of $p$-adic $L$-functions via overconvergent cohomology. Conditional on the non-abelian Leopoldt conjecture, we illustrate them by constructing $p$-adic $L$-functions attached to families of base-change automorphic representations for $\mathrm{GL}(2)$ over CM fields. As a corollary, we prove a $p$-adic Artin formalism result for base-change $p$-adic $L$-functions. 
\end{abstract}

\setcounter{tocdepth}{2}
\footnotesize\tableofcontents \normalsize
\setlength{\parskip}{3pt}

\section*{Introduction}

Let $\pi$ be an automorphic representation of a reductive group $G$, and let $L(\pi,s)$ be an attached $L$-function. The Bloch--Kato conjecture for $L(\pi,s)$ relates its special values to arithmetic data attached to $\pi$, providing a deep link between analysis and arithmetic. Much of what we know about the Bloch--Kato conjectures has come from recasting them into a $p$-adic setting, relating arithmetic data to a \emph{$p$-adic $L$-function} $L_p(\pi)$, which is a $p$-adic analogue of $L(\pi,s)$.

We first recall some general expectations about $p$-adic $L$-functions. For clarity, in this introduction we restrict to a special case. Let $E$ be a number field of degree $r+2s$, where $r$ (resp.\ $s$) is the number of real (resp.\ complex) places of $E$, and let $G = \mathrm{Res}_{E/\Q}\GL_n$. Let $\pi$ be an automorphic representation of $G(\A)$, and $L(\pi,\varphi)$ its standard $L$-function (with the variable $\varphi : E^\times\backslash\A_E^\times \to \C^\times$ ranging over Hecke characters of $E$). We write $\Lambda(\pi,\varphi)$ for the $L$-function completed with the factors at infinity. Then we expect:
\begin{expectation-intro}
	\begin{enumerate}\setlength{\itemsep}{5pt}
		\item There exists a set of \emph{critical Hecke characters} $J$, a set of signs $S_\pi \subset \{\pm 1\}^{2r},$ and a set of \emph{periods} $\{ \Omega_\pi^\epsilon \in \C^\times : \epsilon \in S_\pi\}$, such that for any $\varphi \in J$ of `type $\epsilon$', we have
		\[
		\frac{\Lambda(\pi,\varphi)}{\Omega_\pi^\epsilon} \in \overline{\Q}.
		\]
		Here $S_\pi = \{\pm 1\}^{2r}$ if $n$ is even, yielding $2r$ periods; and it is a singleton set if $n$ is odd, giving only one (see e.g. \cite[\S1]{RS08}). `Being of type $\epsilon$' depends only on $\varphi_\infty$ and $\varphi_{f}(-1)$.
		
		\item\label{item:int} If $J$ is non-empty, then we can $p$-adically interpolate the algebraic values $\Lambda(\pi,\varphi)/\Omega_\pi^\epsilon$. By class field theory, if $\varphi$ is a Hecke character of $E$ of conductor dividing $(p^\infty)$, then there is a corresponding character $\varphi_{(p)}$ on $\mathrm{Gal}_p(E),$ the Galois group of the maximal abelian extension of $E$ unramified outside $p$ and $\infty$. Then there exists a $p$-adic distribution $L_p(\pi)$ on $\mathrm{Gal}_p(E)$, satisfying a good growth condition, such that
		\begin{align*}
			L_p(\pi,\varphi_p) &\defeq \int_{\mathrm{Gal}_p(E)} \varphi_{(p)}(x) \cdot L_p(\pi)\\
			&= (*) \frac{\Lambda(\pi,\varphi)}{\Omega_\pi^\epsilon},
		\end{align*}
		for any $\varphi \in J$ with $p$-power conductor, where $(*)$ is an explicit scaling factor including the Euler factors at $p$.
		
		\item\label{item:families} The distribution $L_p(\pi)$ varies rigid-analytically as $\pi$ varies in a $p$-adic family over an eigenvariety. (We describe precisely what this means in \S\ref{sec:families} and Theorem \ref{thm:intro} below).
		
	\end{enumerate}
\end{expectation-intro}

\begin{remarks*}
	\begin{enumerate}[--]\setlength{\itemsep}{0pt}
		\item Membership of the set $J$ corresponds to Deligne's criterion for critical values of motivic $L$-functions in \cite{Del79}, and can be characterised explicitly in this setting. There is a set $J_\infty \subset \Z^{r+2s}$ of \emph{critical infinity types}, and a Hecke character can be critical only if it has infinity type $\infty(\varphi) \in J_\infty$. If $n$ is even, this is also sufficient; if $n$ is odd, then one must also take a parity condition into consideration.
		\item In practice, before we can complete (2), we must add $p$ to the level (if it is not already present) via a choice of $p$-refinement (or stabilisation). This has the effect of deleting the Euler factors at $p$, which cannot possibly vary $p$-adically, due to the presence of a $p^s$ term. The $p$-adic $L$-function $L_p(\pi)$, and its interpolation and growth properties, will depend crucially on this choice of refinement. We elaborate further in \S\ref{sec:set-up}.
	\end{enumerate}
\end{remarks*}

In practice, this program is hard to carry out. Part (1) is a fairly classical question which is difficult in itself. There remain, however, many cases where we know (1) but not (2); and even further where we know (2) but not (3). For example, for $n \geq 3$, we have a very poor understanding even for $E=\Q$. In the case $n=3$, let $\pi$ be a cohomological, trivial weight, $p$-ordinary automorphic representation that is genuinely for $\GL_3$ (as opposed to, say, a symmetric square transfer); then (2) is the expectation that $L_p(\pi)$ is a $p$-adic \emph{measure} (a \emph{bounded} distribution). There remains, however, no construction of such an object\footnote{Such a construction will be the topic of forthcoming work of the second author with David Loeffler. A construction of an \emph{unbounded} distribution interpolating the $L$-values of $\pi$ was the main result of \cite{Mah00}.}. When $E$ is not imaginary quadratic or totally real, there is an additional complication: it is not at all clear what growth condition $L_p$ should satisfy in the non-ordinary case.

In this paper, in \S\ref{sec:survey} we give a survey of a method for pursuing this program that has given very general results\footnote{See \cite{PS11,PS12,Bel12} for classical modular forms; \cite{Bar15,BDJ17,BH17} for Hilbert modular forms; \cite{Wil17,BW18} for Bianchi modular forms; and \cite{BW_CJM} for general number fields.} in the context of $\GL_2$, that of \emph{overconvergent cohomology}. An advantage of this approach is that, unlike some more classical methods, it sees no difference between ordinary and non-ordinary cases, instead using strictly weaker `small slope' assumptions. 

Further, it is possible to use overconvergent cohomology to construct $p$-adic weight families \cite{AS08,Urb11,Han17}, which means the method is particularly well-adapted to approach (3). We illustrate this in \S\ref{sec:CM families}, where -- assuming the non-abelian Leopoldt conjecture -- we give new cases of (3) in the case of the base-change of Hilbert modular forms to a CM extension, using the construction of (2) given in our previous paper \cite{BW_CJM}. More precisely, let $f^+$ be a cuspidal Hilbert eigenform over a totally real field $E^+$, let $E/E^+$ be a CM extension, and let $f$ be the base-change of $f^+$ to $E$. We assume both $f^+$ and $f$ are non-critical in the sense of Definition \ref{def:non-critical}, which is implied by being small slope (Theorem \ref{thm:small-slope}); in particular, they give rise to points $x_f$ and $x_{f^+}$ in respective eigenvarieties. Let $\sX(\Gal_p(E))$ denote the rigid space of $p$-adic characters on $\Gal_p(E)$.

\begin{theorem-intro}\label{thm:intro}
	Assume the non-abelian Leopoldt conjecture (Conjecture \ref{conj:dim}). In any one-dimensional subvariety of the Hilbert eigenvariety through $x_{f^+}$, lying over a smooth rigid curve in weight space, there exists a neighbourhood $\cV^+$ and a rigid analytic function
	\[
	L_p : \cV^+ \times \sX(\Gal_p(E)) \longrightarrow \Cp
	\]
	satisfying the following interpolation property:
	\begin{itemize}\setlength{\itemsep}{0pt}
		\item for any classical point $y \in \cV^+$, corresponding to a Hilbert modular form $g_y^+$ with base-change $g_y \defeq g_y^+/E$,
		\item and for any Hecke character $\varphi$ of $E$ such that $L(g_y,\varphi)$ is a critical value,
	\end{itemize}
	we have
	\[
	L_p(y,\varphi_{(p)}) = c_{g_y}\cdot (*)\cdot \frac{\Lambda(g_y,\varphi)}{\Omega_{g_y}},
	\]
	for $(*)$ the same interpolation factor as above and $c_{g_y}$ a $p$-adic period.
\end{theorem-intro}
In particular, this holds for $y = x_f$. We define the notion of critical Hecke character, and make the interpolation terms precise, in Theorem \ref{p:admissible}. The non-abelian Leopoldt conjecture is explained in \S\ref{sec:leopoldt}, and is required to control the degree of overconvergent cohomology in which families appear.

In \S\ref{sec:artin formalism}, we use this to prove a $p$-adic Artin formalism result; namely, that the $p$-adic $L$-function of a base-change form factors in a manner analogous to the classical $L$-function after restriction to the cyclotomic variable. We make this precise in Theorem \ref{thm:artin formalism}.\\

\textbf{Acknowledgements:} We thank the organisers of Escuela Latinoamericana de Geometr\'{i}a Algebraica (ELGA), and the Universidad de Talca, for providing a stimulating place to do mathematics; much of this paper was written there. We also thank Pak-Hin Lee and the anonymous referee for their valuable comments and corrections on an earlier version of the paper. DBS was supported by the FONDECYT PAI 77180007. CW was funded by an EPSRC postdoctoral fellowship EP/T001615/1.

\renewcommand{\thesection}{\Roman{section}}


\section{Survey of the construction of $p$-adic $L$-functions}\label{sec:survey}

In this section, we give a survey of the construction of $p$-adic $L$-functions using overconvergent cohomology, highlighting the key difficult steps. The construction is summarised in Figure 1, which we will explain row by row. We maintain notation from the introduction.

\begin{figure}\centering
\begin{tikzcd}[column sep = tiny,row sep = large]
	\mathrm{[T]}&&&\Phi_{\cV}\arrow[rrrrrrrrrr, mapsto, bend left=18]\arrow[dd, mapsto] &\in & \hc{d}(Y_K,\sD_\Sigma)\arrow[dd, "\mathrm{sp}_\lambda"] \arrow[rrrrrr, "\mathrm{Ev}_\Sigma"] &&&&& & \cD(\mathrm{Gal}_p(E),\cO(\Sigma))\arrow[dd, "\mathrm{sp}_\lambda"] & \ni & L_p(\cV)\arrow[dd,mapsto]\\
&&&&&&&&&&&&&\\
\mathrm{[M]}&&& \Phi_{\refpi} \arrow[rrrrrrrrrr, bend left=18, mapsto] & \in & \hc{d}(Y_K,\sD_\lambda)\arrow[dd, "\rho"] \arrow[rrrrrr, "\mathrm{Ev}_\lambda"] &&&& && \cD(\mathrm{Gal}_p(E),\overline{\Q}_p)\ar[dd, "\text{eval.\ at }\varphi_{(p)}"] & \ni & L_p(\refpi)\arrow[dd,mapsto]\\
&&&&&&&&&&&&&\\
\mathrm{[B]} &&&\phi_{\refpi} \arrow[rrrrrrrrrr, bend left=18, mapsto]\arrow[uu,dashed,"\rotatebox{90}{\tiny\text{control thm}}\normalsize"]  &\in &\hc{d}(Y_K,\sV_\lambda^\vee) \arrow[rrrrrr,"\mathrm{Ev}_\varphi"] &&&&& & \overline{\Q}_p & \ni &(*)\frac{\Lambda(\pi,\varphi)}{\Omega_\pi^\epsilon}\\
&&&&&&&&&& \refpi\arrow[lllllllu,dotted]\arrow[rrru,dotted] &&&&&
\end{tikzcd}
\caption{\label{main diagram}\emph{Strategy for constructing $p$-adic $L$-functions}}
\end{figure}

\subsection{Set-up}
\label{sec:set-up}

Throughout this paper, fix an embedding $\overline{\Q} \hookrightarrow \overline{\Q}_p$. Let $G/\Q$ be a reductive group with a fixed model over $\Z$ (which we will also denote $G$); the example of the introduction is $G = \mathrm{Res}_{E/\Q}\GL_n$. Write $B = TN$ for the Borel, where $T$ is the torus and $N$ the unipotent subgroup. Let $B^- = N^-T$ denote the opposite Borel. Let $\A = \R \times \A_f$ denote the ring of adeles (of $\Q$). 

Let $\pi$ be a (regular, algebraic, cuspidal) automorphic representation of $G(\A)$.
Let $K\tame$ be the largest open compact subgroup of $G(\A_f)$ for which $\pi$ admits $K$-fixed vectors. We require specific level at $p$; let
\[
I_{p} \defeq \{g \in G(\Zp) : g\newmod{p} \in B(\F_p)\}
\]
denote the Iwahori. When $G = \GL_n$, $I_p$ is group of matrices that are upper-triangular modulo $p$. If $K\tame_p \not\subset I_p$, before $p$-adic interpolation is possible, we must take a \emph{$p$-refinement} or \emph{stabilisation} of $\pi$. This is a choice, for each $\pri|p$, of an eigenvector in $\pi_{\pri}^{I_p\cap K\tame_{\pri}}$ for a Hecke operator $U_{\pri}$, and yields a collection of $U_{\pri}$-eigenvalues $(\alpha_{\pri})_{\pri|p}$. For $\GL_n$, such a choice is equivalent to choosing an ordering of the Satake parameters.

For $\GL_2$, there is only one sensible choice of operator $U_{\pri}$, but for more general groups the choice of $U_{\pri}$ can be subtle. It will be fundamental to all of our constructions, and we comment further on the choice in \S\ref{sec:control}.

Note in particular that the choice of $p$-refinement determines a cuspidal Hecke eigenform $f \in \pi^{K(\refpi)}$, where $K(\refpi)_v = K\tame_v$ for all $v\nmid p$ and $K(\refpi)_{\pri} =  I_p\cap K\tame_{\pri}$ for all $\pri|p$. We henceforth assume that the choice of $p$-refinement, the corresponding eigenform, and the $U_{\pri}$-eigenvalues $\alpha_{\pri}$ are fixed, and denote all of this data by $\refpi$. We usually just write $K$ in place of $K(\refpi) \subset G(\A_f)$.

\begin{example}
	If $G = \GL_2/\Q$, then $B$ comprises upper-triangular matrices, and $T$ is the diagonal. If an automorphic representation $\pi$ is generated by a newform $f_{\mathrm{new}} \in S_k(\Gamma_0(N))$, then $K\tame$ corresponds to the level $\Gamma_0(N)$, and $K\tame_p = \{\smallmatrd{a}{b}{c}{d} \in G(\Zp): c \equiv 0 \newmod{p^r}\}$, where $p^r || N$. We have two cases:
	\begin{itemize}
		\item[(a)] If $r > 0$, then $K\tame_p \subset I_p$, and there is nothing to be done. Note that $\pi_p^{I_p\cap K\tame_p}$ is in this case a line, and there is a single choice of eigenform (and eigenvalue $\alpha$).
		\item[(b)] If $r = 0$, then $\pi_p^{I_p \cap K\tame_p} = \pi_p^{I_p}$ is two-dimensional, and the characteristic polynomial of $U_p$ on this is $X^2 - a_p(f_{\mathrm{new}})X + p^{k-1}.$ There are thus two $p$-refinements, corresponding to the roots $\alpha$ and $\beta$ of this, which are conjectured to be distinct. These correspond to Hecke eigenforms $f_\alpha, f_\beta$ in $S_k(\Gamma_0(Np))$ on which the Hecke operators for $\ell\neq p$ act with the same eigenvalues as $f_{\mathrm{new}}$ (and on which $U_p$ acts as $\alpha, \beta$ respectively).
	\end{itemize}
\end{example}

\subsection{Classical cohomology} \label{sec:classical}
For this method to work, it is crucial that there exist cohomology classes $\phi_\pi$ in the `classical' cohomology, with $p$-adic coefficients, which moreover are eigenclasses for a natural Hecke action with the same eigenvalues as $f$. Such classes are provided by classical automorphic theory, which we now describe. Attached to any open compact subgroup of $K \subset G(\A_f)$, we have an associated locally symmetric space 
\[
Y_K \defeq G(\Q)\backslash G(\A)/K_\infty^\circ K,
\]
where $K_\infty^0$ is a maximal compact-mod-centre subgroup of $G(\R)$. This decomposes into a finite disjoint union of quotients of symmetric spaces by arithmetic subgroups of $G(\Q)$. If $M$ is a right $K$-module, it gives a local system $\sM$ on $Y_K$, defined as the locally constant sections of projection
\begin{equation}\label{eq:local system}
	G(\Q) \backslash[G(\A) \times M]/K_\infty^\circ K \longrightarrow Y_K
\end{equation} 
with action $\gamma(g,m)zu = (\gamma gzu, m|u)$, for $\gamma \in G(\Q), z \in K_\infty^\circ$ and $u \in K$. Note this local system is trivial if the centre $Z(G(\Q)) \cap K$ does not act trivially.

\subsubsection{Classical coefficient sheaves}
The cohomology we consider has coefficients in algebraic local systems, which arise as induced representations of weight characters. A \emph{weight} is a character of the torus $T$; it is \emph{algebraic} if the character is algebraic. The weights we consider may be viewed $p$-adically as characters $T(\Zp) \to \overline{\Q}_p$ in a natural way. The weight of an automorphic representation $\pi$, and whether or not $\pi$ is cohomological, depends only on its infinite component $\pi_\infty$. For the rest of the paper, we will \emph{always} assume $\lambda$ is cohomological.

For a $\Zp$-algebra $R$, and a $p$-adic weight $\lambda$, we define an algebraic representation -- the $R$-points of the representation of $G$ of highest weight $\lambda$ -- via parabolic induction:
\begin{align}\label{eq:algebraic}
	V_\lambda(R) &\defeq \mathrm{Ind}_{B^-(\Zp)}^{G(\Zp)} \lambda\notag\\
	&=  \{\theta : G(\Zp) \rightarrow R : \theta \text{ is algebraic, }  \theta(n^{-}tg) = \lambda(t)\theta(g) \ \ \forall n^-,t,g\},
\end{align}
where $n^- \in N^-(\Zp), t \in T(\Zp)$ and $g \in G(\Zp)$. It is standard that, via restriction, $V_\lambda(R)$ may be identified with the set of functions $\theta : I_p \to R$ satisfying \eqref{eq:algebraic} (e.g.\ \cite[\S3.2.9]{Urb11} with trivial $\epsilon$). We also have the Iwahori decomposition $I_p = I_p^-T(\Zp)N(\Zp)$, where $I_p^- \defeq I_p \cap N^-(\Zp)$. Thus by the transformation property any $\theta \in V_\lambda(R)$ is uniquely determined by its values on the unipotent $N(\Zp)$.

The space $V_\lambda(R)$ is naturally a left-$G(\Zp)$-module by right translation; we then let $u \in K$ act via its projection $u_p$ to $G(\Zp)$. The linear $R$-dual $V_\lambda^\vee(R) \defeq \mathrm{Hom}(V_\lambda(R),R)$ is then a right $K$-module by $\mu|u(\theta) = \mu(u \cdot \theta)$, and we get an associated local system $\sV_\lambda^\vee(R)$. We similarly have a complex local system $\sV_{\lambda}^{\vee}(\C)$ by instead considering the complex algebraic induction.

\subsubsection{Classical cohomology classes and rationality}
Suppose $\pi$ admits $K$-fixed vectors. Then there exist integers $q,\ell$ such that $\pi$ contributes to the cohomology groups $\hc{i}(Y_K,\sV_\lambda^\vee(\C))$ for $q \leq i \leq q+\ell$.  More precisely, we apply this to $K = K(\refpi)$ from above, and the Hecke eigenform $f \in \pi^K$ corresponding to our choice of $p$-refinement. Then for such $i$, there exist Hecke eigenclasses $\phi_{\refpi,\C} \in \hc{i}(Y_K,\sV_\lambda^\vee(\C))$ with the same eigenvalues as $f$. See \cite{Bor74}, \cite{Bor81} or \cite{Clo90} for more details.

These cohomology groups admit natural rational structure. However, whilst a choice of cusp form $f \in \pi$ determines canonical \emph{complex} cohomology classes, it is more subtle to obtain analogous classes in the rational cohomology. In desirable situations, after dividing by appropriate complex periods $\Omega_\pi^\epsilon$, we can construct a Hecke eigenclass $\phi_{\refpi,\overline{\Q}}$ with coefficients in $V_\lambda^\vee(\overline{\Q})$, which we may then consider to have $p$-adic coefficients, yielding a class $\phi_{\refpi} \in \hc{i}(Y_K,\sV_\lambda^\vee(\overline{\Q}_p))$. The periods will depend on the degree of cohomology we work in.

\begin{example}\label{ex:gl2 weights}
	Suppose $G = \GL_2/\Q$. Then algebraic dominant weights are pairs $(\lambda_1,\lambda_2) \in \Z^2$ with $\lambda_1 \geq \lambda_2$, corresponding to the character $\mathrm{diag}(t_1,t_2) \mapsto t_1^{\lambda_1}t_2^{\lambda_2}$ of $T$. We can normalise by $\mathrm{det}^{-\lambda_2}$ to give a weight indexed by a single integer $\lambda = (k,0) = (\lambda_1-\lambda_2, 0)$. A ($p$-refined) automorphic representation $\refpi$ of weight $(k,0)$ corresponds to a ($p$-refined) weight $k+2$ modular eigenform $f$, and the locally symmetric spaces are modular curves. The representation $V_{(k,0)}$ is the space of polynomials of degree at most $k$; in this case, cuspidal modular forms contribute to the cohomology $\hc{i}$ only in degree 1.
	
	The rationality in this case is a consequence of strong multiplicity one. The Hecke eigenspaces $\hc{1}(Y_K,\sV_\lambda^\vee(L))^\pm[\refpi]$ -- where the Hecke operators act as they do on $f$, and where $\smallmatrd{-1}{0}{0}{1}$ acts as $\pm1$ -- are one-dimensional for $L = \C,\overline{\Q}$. We have a natural decomposition $\hc{1} = (\hc{1})^+ \oplus (\hc{1})^-$, with associated projectors $\mathrm{pr}^\pm$, and the periods $\Omega_f^\pm$ measure the difference between:
	\begin{enumerate}\setlength{\itemsep}{1pt}
		\item the canonical classes $\phi_{\refpi,\C}^\pm \defeq \mathrm{pr}^\pm(\phi_{\refpi,\C})$ attached to a normalised eigenform,
		\item and (a choice of) generators $\phi_{\refpi,\overline{\Q}}^\pm$ for the line $\hc{1}(Y_K,\sV_\lambda^\vee(\overline{\Q}))^\pm[\refpi].$
	\end{enumerate}
	We then define $\phi_{\refpi}^\pm = \phi_{\refpi,\overline{\Q}}^\pm = \phi_{\refpi,\C}^\pm/\Omega_\pi^\pm$, and $\phi_{\refpi} \defeq \phi_{\refpi}^+ + \phi_{\refpi}^-.$ Note the periods are only well-defined up to algebraic scaling, depending on the choice of algebraic generators.
	
\end{example}

\begin{example}
	If $G = \mathrm{Res}_{E/\Q}\GL_2$ where $E$ is an imaginary quadratic field, then we are in the setting of \emph{Bianchi} modular forms, and such forms contribute to both $\hc{1}$ and $\hc{2}$. There are no signs to consider here, but we do get periods $\Omega_\pi^1$ and $\Omega_\pi^2$ arising from degree 1 and 2 cohomology respectively. These periods are expected to be related, but genuinely different: for example, \cite[Conj.\ 1]{TU18} conjectures a precise relation between $\Omega_\pi^1$ and $\Omega_\pi^2$ for base-change Bianchi modular forms.
\end{example}

\subsection{Evaluation maps}\label{sec:abstract evaluation}
The horizontal rows of Figure \ref{main diagram} are given by \emph{evaluation maps}. We describe one general strategy for constructing them and highlight areas of subtlety. 

\begin{lemma}\label{lem:top degree trivial}
	Let $X$ be a connected orientable real manifold of dimension $d$, and let $\sN$ be a trivial local system (that is, there exists a module $N$ such that $\sN$ is given by locally constant sections of the projection $X\times N \rightarrow X$). Then 
	\[
	\hc{d}(X,\sN) \cong N,
	\]
	the isomorphism given by pairing against some generator of $\h_d^{\mathrm{BM}}(X,\Z) \cong \Z$.
\end{lemma} 

Let $M$ be a right $K$-module, and consider the following construction strategy.
\begin{construction}\label{cons:eval}
	\begin{enumerate}\setlength{\itemsep}{0pt}
		\item Find a reductive group $H$ with an embedding $\iota : H \hookrightarrow G$, and locally symmetric spaces $X_c$ attached to $H$ that have dimension
		$d \in [q,q+\ell]$ and `conductor' $c$.
		\item By considering appropriate twists of $\iota$ and $\sM$ by `conductor $c$', construct a map
		\[
		\widetilde{\iota}_c : \hc{d}(Y_K,\sM) \longrightarrow \hc{d}(X_c,\iota^*\sM).
		\]
		
		\item Decompose $X_c = \bigsqcup_a X_c^a$ over its group $\pi_0(X_c)$ of connected components, which will usually be a narrow ray class group. Each $X_{c,a} = \Gamma_{c,a}\backslash \cX$ is a quotient of a symmetric space by an arithmetic group. Thus we obtain
		\[
		\hc{d}(X_c,\iota^*\sM) \cong \bigoplus_a \hc{d}(\Gamma_{c,a}\backslash\cX,\iota_{x_a}^*\iota^*\sM),
		\]
		where $\iota_{x_a}$ is the inclusion $\Gamma_{c,a}\backslash\cX \hookrightarrow X_c$, which -- along with $\Gamma_{c,a}$ -- will depend on a choice of adelic representative $x_a$ of $a$.
		
		\item We consider the $\Gamma_{c,a}$-coinvariants of $M$ denoted by $M_{\Gamma_{c,a}}$, and the quotient map $M \rightarrow M_{\Gamma_{c,a}}$. This then gives a trivial local system $\sM_{c,a}$ on $\Gamma_{c,a}\backslash \cX$. Pushing forward under this gives a map
		\[
		\mathrm{coinv}_{c,a} : \hc{d}(\Gamma_{c,a}\backslash\cX, \iota_{x_a}^*\iota^*\sM) \rightarrow \hc{d}(\Gamma_{c,a}\backslash\cX,\sM_{c,a}).
		\]
		
		\item Integrate against a fixed choice of cycle as in Lemma \ref{lem:top degree trivial}, giving an isomorphism
		\[
		\mathrm{triv} : \hc{d}(\Gamma_{c,a}\backslash\cX,\sM_{c,a}) \isorightarrow M_ {\Gamma_{c,a}}.
		\]
	\end{enumerate}
\end{construction}

The evaluation map is then the composition
\[
\mathrm{Ev}_{c,x_a}^M : \hc{d}(Y_K, \sM) \longrightarrow M_{\Gamma_{c,a}}
\]
given by $\mathrm{triv}\circ\mathrm{coinv}_{c,j} \circ \iota_{x_a}^* \circ \widetilde{\iota}_c.$ \emph{A priori} each \emph{individual} evaluation map will depend, in a controlled way, on the choice of representatives. In the applications, this dependence ultimately disappears (see Remark \ref{rem:independent}).

\begin{remark}\label{rem:covariance}
	Where this process can be made to work, it is covariant in the local system $\sM$; that is, if $\nu : M \rightarrow N$ is a map of $K$-modules, then $\nu$ descends to a map between the coinvariants, and we have $\nu\circ\mathrm{Ev}^M_{c,x_a} = \mathrm{Ev}_{c,x_a}^N\circ \nu_*.$
\end{remark}

\begin{example}
	In the case of modular forms, i.e.\ $G = \GL_2/\Q$, we have $Y_K = \Gamma_K\backslash \uhp$. Let $a, c \in \Z$, and $X_{c,a}'$ be the vertical path between the cusps $a/c$ and $\infty$ in $\uhp$.  We write $X_{c,a}$ for the image of $X_{c,a}'$ in $Y_K$. The component group is $(\Z/c)^\times \cong \cl_{\Q}^+(c)$ (the narrow ray class group of conductor $c$), and the cycle $X_c$ is the collection $\bigsqcup_{a\in(\Z/c)^\times} X_{c,a}$. In this case the relevant cohomology group is $\mathrm{H}^1_{\mathrm{c}}$, which is isomorphic to the more concrete space of modular symbols (e.g. \cite{PS11}); then it can be shown that $M = M_{c,a}$, and $\mathrm{Ev}_{c,x_a}^M$ is the same as evaluating a modular symbol at the pair of cusps $\{a/c - \infty\}$. 
\end{example}

\begin{remark}
	Sometimes the space $X_c$ will not be a typical locally symmetric space attached to $H$. In particular if the map $\iota : H \hookrightarrow G$ does not map the centre $Z_H$ into $Z_G$, then it does not induce an honest map on locally symmetric spaces, and instead we should consider quotients by $\iota^{-1}(Z_G)$. Additionally this can help the spaces have the correct dimension; for example, consider the case $H = \GL_n \times \GL_n \subset \GL_{2n}$. The locally symmetric spaces attached to $H$ are defined using a quotient by the centre of $H$; in the case $n=1$, this gives a 0-dimensional space, instead of the 1-dimensional space we need. Modifying to instead quotient by the centre of $G$, we get spaces of the correct dimension (see \cite{GR2}).
\end{remark}

\begin{remark}
	Whilst we have highlighted this particular approach to constructing evaluation maps, there are alternatives in the literature; and in particular, there are many cases where such maps can be constructed at classical level by pairing with Eisenstein cohomology classes. The \emph{Rankin--Selberg method} is a useful tool for relating such evaluations to critical $L$-values. It would be very interesting to construct analogues of these evaluation maps for overconvergent cohomology, which would give new constructions of many existing $p$-adic $L$-functions in this framework. In the spirit of this survey, it would also allow clean generalisations of these constructions to non-ordinary settings (via \S\ref{ss: survey overconvergent cohomology}) and variation in $p$-adic families (via \S\ref{sec:top row II}).
\end{remark}

\subsection{Bottom row: classical computations}
\label{sec:bottom row}
In the bottom row, we consider evaluation maps for $M = V_\lambda^\vee$. For a critical Hecke character $\varphi$, the aim is to construct maps
\begin{align*}
	\mathrm{Ev}_\varphi : \hc{i}(Y_K,\sV_\lambda^\vee) &\longrightarrow \overline{\Q}_p\\
	\phi_{\refpi} &\longmapsto (*)\Lambda(\pi,\varphi)/\Omega_{\pi}^\epsilon,
\end{align*}
giving the required connection to classical critical $L$-values. Here $(*)$ is an explicit interpolation factor that will depend on the choice of $p$-refinement $f \in \pi$ to which we attach $\phi_{\refpi}$ (in particular, it will contain Euler factors for each prime above $p$ at which $f$ is an oldform).

Crucial to this construction is the existence of a distinguished vector $v_j \in V_\lambda$ attached to the critical infinity type $j$ of $\varphi$. Evaluation at $v_j$ gives a map $\mathrm{ev}_{v_j}: V_\lambda^\vee(\overline{\Q}_p) \to \overline{\Q}_p$ of $K$-modules for an appropriate action of $K$ on $\overline{\Q}_p$ (typically by a determinant character depending on $j$). If $\varphi$ has conductor $c$, then via a determinant map on $H$, we can usually interpret it as a character on the component group $\pi_0(X_c)$. We can then define an expression of the form
\[
\mathrm{Ev}_\varphi \defeq \sum_{a \in \pi_0(X_c)} \varphi(x_a) \left[\mathrm{ev}_{v_j}\circ \mathrm{Ev}^{V_{\lambda}^\vee}_{c,x_a}\right].
\]
This should now be independent of the choice of representatives $x_a$. One then hopes to relate $\mathrm{Ev}_\varphi(\phi_{\refpi})$ to a period integral computing critical $L$-values of $\pi$. In order to have a hope of executing this strategy it is crucial to have algebraicity results for critical $L$-values of $\pi$ based upon period integrals over a group $H\subset G$, which do not involve an Eisenstein series; for example, in the case of $G = \GL_n \times \GL_{n-1}$, where critical $L$-values can be computed from diagonally embedding $H = \GL_{n-1}$. One should then be able to construct evaluation maps based on cycles obtained from $H$, and via the overconvergent cohomology approach, be able to recover/generalise results of Januszewski \cite{Jan11,Jan15}.

\begin{example}
	In the case of classical modular forms, $V_\lambda = V_k$ is a space of polynomials of degree at most $k$, and the critical infinity types are $0,...,k$. The distinguished vector $v_j$ is the monomial $x^j$. If $f$ is a classical modular form of weight $k+2$, then the coefficients of $\mathrm{Ev}_{c,x_a}^{V_k^\vee}(\phi_f)$ relate to period integrals of $f$ over the paths $a/c \to \infty$, and it is a standard computation that a $\chi$-weighted sum of such period integrals computes the twisted $L$-values $L(f,\chi,j+1)$ for $0\leq j\leq k$ (for example, see \cite[\S2.8]{Pol11}).
\end{example}

\subsection{Middle row: overconvergent cohomology}\label{ss: survey overconvergent cohomology}
Suppose now we have constructed abstract evaluation maps as in \S\ref{sec:abstract evaluation} and used them to construct the bottom row [B] of Figure \ref{main diagram}, as outlined in \S\ref{sec:bottom row}. 

\subsubsection{Definitions and strategy}

A $p$-adic $L$-function should $p$-adically interpolate the critical $L$-values in the right-hand side of row [B]. In order to do this, a beautiful insight of Stevens \cite{Ste94}, later explored in joint work with Pollack \cite{PS11}, was that we could instead try to $p$-adically interpolate the evaluation maps in row [B]. This requires his theory of \emph{overconvergent cohomology}, comprising the left-hand side of Figure \ref{main diagram}. The key idea is to replace the finite-dimensional algebraic module $V_\lambda$ of polynomial functions with an infinite dimensional module $\cA_\lambda$ of analytic functions.

Such interpolation of coefficients is only possible after adding $p$ to the level; this is required to define an action of the level group on $\cA_\lambda$, and hence to obtain attached local systems. Recall the definition of $I_p$ from \S\ref{sec:set-up}.

\begin{definition}
	Let $L/\Qp$ be a field extension, and define the \emph{locally analytic induction}
	\begin{align*}
		\cA_\lambda(L) &\defeq \mathrm{LA}\cdot\mathrm{Ind}_{B^-(\Zp)\cap I_p}^{I_p} \lambda\\
		&= \{\theta : I_p \rightarrow L :\theta\text{ locally analytic, } \theta(n^{-}tg) = \lambda(t)\theta(g) \ \ \forall n^-,t,g\}.
	\end{align*}
	Then define 
	\[
	\cD_\lambda(L) \defeq \mathrm{Hom}_{\mathrm{cts}}(\cA(L),L)
	\]
	to be the continuous dual of \emph{locally analytic distributions} for $G$.
\end{definition}

This definition is explored in more detail in \cite[\S3.2]{Urb11}. Again, these functions are determined by their values on the unipotent subgroup $N(\Zp)$, and we require the restriction to this to be locally analytic (after identifying $N(\Zp)$ with an open subspace of $\Qp^r$ for some $r$).

\begin{example} 
	For $\GL_2/\Q$, the unipotent $N(\Zp)$ is the group of $\smallmatrd{1}{x}{0}{1}$ with $x \in \Zp$, so the definition reduces to functions $\theta : \Zp \to L$. Such a $\theta$ is \emph{locally analytic} if for every $x \in \Zp$, there exists an $n$ such that $\theta$ can be written as a convergent power series on $x + p^n\Zp$.
\end{example}

The space $\cA_\lambda(L)$ can be equipped with a left action of $I_p$ by $\gamma \cdot \theta(g) = \theta(g\gamma)$. For level subgroups $K$ with $K_p \subset I_p$, it is thus a left $K$-module. The dual $\cD_\lambda(L)$ inherits a right $K$-module structure, and we get a corresponding local system $\sD_\lambda(L)$.

We have a natural inclusion $V_\lambda(L) \subset \cA_\lambda(L)$; in the case of $\GL_2/\Q$, this can be seen as polynomial functions are in particular locally analytic. This dualises to a surjection
\[
\cD_\lambda(L) \longrightarrow V_\lambda^\vee(L).
\]
This induces the map
\[
\rho : \hc{i}(Y_K,\sD_\lambda(L)) \longrightarrow \hc{i}(Y_K, \sV_\lambda^\vee(L))
\]
in Figure \ref{main diagram}. This is much closer to our goal, to obtain a distribution on $p$-adic Hecke characters: by definition, the coefficient modules are now distributions on $p$-adic functions. Let $d$ be the degree of cohomology for which we have evaluation maps as in \S\ref{sec:abstract evaluation}. The strategy now becomes:
\begin{strategy}\label{strategy}
	\begin{itemize}\setlength{\itemsep}{0pt}
		\item[(1)] Exhibit a distinguished class $\Phi_{\refpi} \in \hc{d}(Y_K,\sD_\lambda)$ such that $\rho(\Phi_{\refpi}) = \phi_{\refpi}$;
		\item[(2)]  Let $c$ be an ideal with $(p)|c|(p^\infty)$. Use the evaluation maps of \ref{sec:abstract evaluation} to build a map
		\[
		\mathrm{Ev}^c_\lambda : \hc{d}(Y_K,\sD_\lambda) \longrightarrow \cD(\mathrm{Gal}_p(E),\overline{\Q}_p)
		\]
		such that for any critical Hecke character $\varphi$ of conductor $c$, we have $\mathrm{ev}_{\varphi_{(p)}} \circ \mathrm{Ev}_\lambda^c = \mathrm{Ev}_\varphi \circ \rho$ as maps $\hc{d}(Y_K,\sD_\lambda) \to \overline{\Q}_p$. Here $\varphi_{(p)}$ is the character on $\mathrm{Gal}_p(E)$ corresponding to $\varphi$.
		
		\item[(3)] Prove that the $\mathrm{Ev}_\lambda^c$ are compatible in changing $c$: typically this will take the form $\mathrm{Ev}_\lambda^{c\pri} = \mathrm{Ev}_\lambda^c \circ U_{\pri}$. This allows us to make sense of `$\mathrm{Ev}_\lambda \defeq U_c^{-1} \mathrm{Ev}_\lambda^c$' on finite-slope eigenclasses, which is then independent of $c$. (This is row [M] of Figure \ref{main diagram}).
	\end{itemize}
\end{strategy}

Given the above, we then define $L_p(\refpi) \defeq \alpha_p^{-1}\mathrm{Ev}^{(p)}_\lambda(\Phi_{\refpi})$ where $\alpha_p$ is the eigenvalue of $U_p$. We see that, for $\varphi$ of conductor $c$, with $(p)|c|(p^\infty)$, we have
\begin{align*}
	L_p(\refpi,\varphi_{(p)}) = \alpha_c^{-1} \mathrm{Ev}^c_\lambda(\Phi_{\refpi})(\varphi_{(p)}) &= \alpha_c^{-1}\mathrm{Ev}_\varphi \circ \rho(\Phi_p)\\
	& = \alpha_c^{-1} \mathrm{Ev}_\varphi(\phi_{\refpi}) = (*)\Lambda(\pi,\varphi)/\Omega_\pi^\epsilon,
\end{align*}
that is, $L_p(\refpi) \in \cD(\mathrm{Gal}_p(E),\overline{\Q}_p)$ interpolates the algebraic parts of critical $L$-values.

\subsubsection{Canonical lifts: the control theorem}\label{sec:control}

The overconvergent class $\Phi_{\refpi}$ is given by a cohomological analogue of Coleman's control theorem for overconvergent modular forms. This says that $\rho$ can be controlled after restricting to particular eigenspaces of certain operators. 

\begin{definition}\label{def:controlling}
	Let $t \in T(\Qp)$ such that $t^{-1}N(\Zp)t \subset N(\Zp)$. Conjugation of $N(\Zp)$ by $t$ induces an action on distributions, and we obtain a double coset operator $U_t \defeq [KtK]$ on overconvergent cohomology. We say $U_t$ is a \emph{controlling operator} if 
	\[
	\bigcap_{i \geq 0} t^{-i} N(\Zp) t^{i} = \{1\},
	\]
	in which case $U_t$ acts compactly on overconvergent cohomology \cite[Lem.\ 3.2.8]{Urb11}.
\end{definition}
\begin{example}
	In the case $G= \mathrm{Res}_{E/\Q}\GL_n$ for $E$ a number field, the matrix $t = \mathrm{diag}(1,p,..., p^{n-1})$ satisfies this condition, and gives rise to the (controlling) Hecke operator $U_p = U_t$.
\end{example}

\begin{theorem}\label{thm:control}
	Let $U_t$ be a controlling operator. There exists $C_\lambda(t) > 0$ such that, for all $\alpha \in L^\times$ with $v_p(\alpha) < C_\lambda(t)$ and $i \in \N$, the restriction of $\rho$ to the generalised $\alpha$-eigenspaces 
	\[
	\rho : \hc{i}(Y_K,\sD_\lambda)^{U_t = \alpha} \isorightarrow \hc{i}(Y_K,\sV_\lambda^\vee)^{U_t = \alpha}
	\]
	of the $U_t$-operator is an isomorphism.
\end{theorem}

In practice, this means that a unique overconvergent lift $\Phi_{\refpi}$ of $\phi_{\refpi}$ is guaranteed by a simple numerical criterion on a Hecke eigenvalue at $p$. Precise versions of this theorem in the general setting (though for singular cohomology) are given in \cite{AS08} and \cite{Urb11}.

\begin{example}
	For $\GL_2/\Q$ and $\lambda = (k,0)$, and controlling operator $U_p$, we have $C_\lambda = k+1$. Thus for any normalised weight $k+2$ cuspidal eigenform $f$ with $v_p(a_p(f)) < k+1$, the restriction of $\rho$ to the $f$-eigenspaces is an isomorphism.
\end{example}

\begin{remarks}\label{rem:multi}
	\begin{itemize}\setlength{\itemsep}{0pt}
		\item[(i)] This theorem implies that if the $U_t$-eigenvalues of $\refpi$ satisfy $v_p(\alpha) < C_\lambda(t)$, then the restriction of $\rho$ to the $\tilde{\pi}$-eigenspaces is an isomorphism. The converse is not true. Even in the case of $\GL_2/\Q$, there are examples of modular forms with $v_p(a_p(f)) = k+1$ such that the restriction of $\rho$ to the $f$-eigenspaces is an isomorphism (see \cite{PS12,Bel12}).
		\item[(ii)] Stronger versions of this theorem are sometimes available. If we can factor the $U_p$ operator into products of `smaller' operators, we can lift with respect to these instead. For example, if $E/\Q$ is an imaginary quadratic field in which $p$ splits as $\pri\overline{\pri}$, then we get an isomorphism when $v_p(a_{\pri}(f)), v_p(a_{\overline{\pri}}(f)) < k+1$ (see \cite[\S6]{Wil17}).
	\end{itemize}
\end{remarks}

\subsubsection{Galois evaluation maps}
\label{sec:galois evaluations}
To enact part (2) of Strategy \ref{strategy}, we study $\mathrm{Gal}_p(E)$ more closely. By class field theory, it admits as a dense subgroup the narrow ray class group 
\[
\cl_E^+(p^\infty) \defeq E^\times\backslash\A_E^\times/E_\infty^+\widehat{\cO}_E^{\times,(p)},
\]
where $E_\infty^+$ is the connected component of the identity in $(E\otimes\R)^\times$ and $\widehat{\cO}_E^{\times,(p)} = \prod_{v\nmid p}\cO_{E_v}^\times$ are the integral ideles away from $p$. For any ideal $c|(p^\infty)$ in $\cO_E$, this fits into an exact sequence
\[
1 \to \frac{(1 + c\cO_p)}{\overline{\cO}_{E,c}^{\times,+}}\longrightarrow \cl_E^+(p^\infty) \xrightarrow{\mathrm{pr}}\cl_E^+(c) \to 1,
\]
where $\cO_p = \cO_E\otimes\Zp$, $\overline{\cO}_{E,c}^{\times,+}$ is the $p$-adic completion of $\{\epsilon \in \cO_E^{\times,+} : \epsilon \equiv 1 \newmod{c}\}$ and $\cl_E^+(c)$ is the narrow ray class group of $E$ of conductor $c$. In particular, to each $y \in \cl_E^+(c)$, we get a corresponding subspace
\[
\mathrm{Gal}_p[y] \defeq \mathrm{pr}^{-1}(\{y\}) \subset \mathrm{Gal}_p(E),
\]
and $\mathrm{Gal}_p(E) = \sqcup_{y} \mathrm{Gal}_p[y].$ In particular, for each such $c$ we also have a decomposition
\[
\cD(\mathrm{Gal}_p(E),L) \cong \bigoplus_{y\in \cl_E^+(c)} \cD(\mathrm{Gal}_p[y], L),
\]
where if $\cG$ is a $p$-adic group, we write $\cD(\cG,L)$ for the space of $L$-valued locally analytic distributions on $\cG$. The importance of this decomposition stems from the observation that in all known cases where this construction has been carried out, for conductors $c$ with $(p)|c|(p^\infty)$, there is a surjection
\[
\mathrm{\xi} : \pi_0(X_c)\longrightarrow \cl_E^+(c).
\] 
One arrives at the following `piecewise' construction of the Galois evaluation map.
\begin{strategy}
	For each $y \in \cl_E^+(c),$ construct a map
	\[
	\beta_y: \cA(\mathrm{Gal}_p[y], L) \rightarrow \cA_\lambda(L)
	\]
	which moreover lands in the subspace $\cA_\lambda(L)^{\Gamma_{c,a}}$ of $\Gamma_{c,a}$-invariants, where $\xi(a) = y$.
\end{strategy}
Then, dualising and composing, we get map
\[
\mathrm{Ev}_\lambda^c[y] : \hc{d}(Y_K,\sD_\lambda) \labelrightarrow{\mathrm{Ev}_{c,x_a}^{\cD_\lambda}}\cD_{\lambda}(L)_{\Gamma_{c,a}} \labelrightarrow{\beta_y^\vee} \cD(\mathrm{Gal}_p[y],L),
\]
which we think of as the `$y$-part' of a candidate for $\mathrm{Ev}_\lambda$. One checks that this is independent of any choice of lift of $y$ under $\xi$, and any choices of representatives. We then essentially define 
\[
\mathrm{Ev}_\lambda^c = \sum_y \mathrm{Ev}_\lambda^c[y],
\]
perhaps up to an additional twisted auxiliary sum over $a \in \xi^{-1}(y)$ (for example, required in the case of $\GL_{2n}$, where this map is averaged over a `Shalika character').

\begin{remark}\label{rem:independent}
	The map $\beta_y$ involves an identification of $\Gal_p[y]$ with a subquotient of $\cO_p$, which involves another adelic choice of representative. This cancels out the choice of $x_a$ and renders the map $\mathrm{Ev}_\lambda^c$ independent of these choices.
\end{remark}

\begin{example}
	For $\GL_2/\Q$, we have $\Galp(\Q) = \Zp^\times$, and $\Galp[y] = y + p^c\Zp$. From above $\cA_\lambda$ is the space of locally analytic functions on $\Zp$, and $\Gamma_{c,a}$ is a subset of the totally positive units in $\Q$, hence trivial (so that $\cA_{\lambda} = \cA_{\lambda}^{\Gamma_{c,a}}$). The map $\beta_y : \cA(\Galp[y]) \to \cA_\lambda$ is then just the natural inclusion, and $\beta_y^\vee$ restriction from $\Zp$ to $y + p^c\Zp$.
\end{example}


\subsubsection{Horizontal compatibility and admissibility}

The above construction now gives a family of candidate Galois evaluation maps
\[
\big\{ \mathrm{Ev}_\lambda^c : \hc{d}(Y_K,\cD_\lambda) \longrightarrow \cD(\Gal_p(E),L) \ : \ (p)|c |(p^\infty)\big\}.
\]
We wish to know how this family varies if we vary $c$. Such results (`horizontal compatibility') are obtained using the $U_{\pri}$-operators for $\pri|p$ in $E$. In particular, we typically have:

\begin{expectation}\label{expectation}
	Let $\pri|p$. Then
	\[
	\mathrm{Ev}_\lambda^{c\pri} = \mathrm{Ev}_\lambda^c \circ U_{\pri}
	\]
	as maps $\hc{d}(Y_K,\cD_\lambda) \to \cD(\mathrm{Gal}_p(E),L)$. Thus if $\Phi$ is an eigenclass for the $U_{\pri}$-operators, then
	\[
	\mathrm{Ev}_\lambda(\Phi) := \alpha_c^{-1}\mathrm{Ev}_\lambda^c(\Phi) \in \cD(\mathrm{Gal}_p(E),L)
	\]
	is independent of $c$ (where $\alpha_c$ is the $U_c$-eigenvalue of $\Phi$).
\end{expectation}

This compatibility with $U_{\pri}$-operators also allows control of the growth of the resulting distributions. In particular, in the totally real and imaginary quadratic case, it gives a bound on the growth depending on the $U_{\pri}$-eigenvalues. The precise dependence will vary with the group $G$, but in the case $\GL_2$, we obtain:

\begin{corollary}
	Suppose $G = \mathrm{Res}_{E/\Q}\GL_2$ where $E$ is a totally real or imaginary quadratic field. If $\Phi$ is an eigenclass for the $U_{\pri}$-operators, then $\mathrm{Ev}_\lambda(\Phi)$ is admissible of order depending on the $\alpha_{\pri}$ for $\pri|p$ (in the sense of, for example, the growth parameters of \cite[\S3]{Loe14}).
\end{corollary}

This then completes Strategy \ref{strategy}, the middle row of Figure \ref{main diagram}, and the construction of the $p$-adic $L$-function at a single weight.

\subsection{Top row, I: formalism of families}\label{sec:families}
The theory of $p$-adic families of automorphic representations is captured geometrically in the study of \emph{eigenvarieties}. As a first step, we encode the variation of weights in $p$-adic geometry, giving a brief survey of the construction of eigenvarieties using overconvergent cohomology, following \cite{AS08,Urb11,Han17}.

\subsubsection{Weight spaces}
Recall that a weight of an automorphic representation of $G(\A)$ is naturally an algebraic character on a maximal torus $T \subset G$, so a \emph{$p$-adic} weight should be a continuous homomorphism $T(\Zp) \to L^\times$, for $L/\Qp$ a field extension. In all of our constructions we must also take the level group into account: let $K$ be as above, and let $\overline{Z}_K$ denote the $p$-adic closure of $Z_{G}(\Q) \cap K$ in $T(\Zp)$. 

\begin{definition}[Weights for $T$]
	Define the \emph{weight space} of level $K$ for $G$ to be the rigid analytic space whose $L$-points, for $L \subset \C_p$ any sufficiently large extension of $\Qp$, are given by
	\[ 
	\cW_K(L) = \mathrm{Hom}_{\mathrm{cts}}\big(T(\Z_p)/\overline{Z}_K,L^{\times}\big).
	\]	
\end{definition}

This space has dimension $\dim T(\Zp) - \dim \overline{Z}_K$. Recall that if $M$ is a $K$-module, then we defined in \eqref{eq:local system} a local system attached to $M$, which is non-trivial only if $Z(G(\Q)) \cap K$ acts trivially; for the modules we later define, this corresponds precisely to the condition that weight characters are trivial on $\overline{Z}_K$. We will freely identify a weight $\lambda \in \cW_K(L)$ with the corresponding character on $T(\Zp)$ that is trivial under $\overline{Z}_K$, and as $K$ will be fixed, we will henceforth drop it from the notation.

\begin{example}
	For $\GL_2/\Q$, the weight of a modular form is an integer $k+2$ and a character $\chi$. This is encoded via the complex Hecke character $z \mapsto |z|^k \chi(z)$, where we lift $\chi$ to the ideles in the usual way. If we instead consider continuous characters $\varphi_p : \Q^\times\backslash\A^\times \to \C_p^\times$, then for topological reasons, $\varphi_p$ must be trivial on $\R_{>0}$ and factors through a finite quotient of $\prod_{\ell\neq p}\Z_{\ell}^\times$. Thus a $p$-adic weight is a finite order twist of a continuous character on $\Zp^\times$. Up to twists by the norm, exactly as in Example \ref{ex:gl2 weights}, we recover a continuous $p$-adic character of $T(\Zp)$. (Note in this case $\overline{Z}_K \subset \{\pm 1\},$ and we recover the familiar fact that the space of modular forms is $0$ if $k$ is odd and $\smallmatrd{-1}{0}{0}{-1}$ is in the level).
\end{example}

Consider an affinoid neighbourhood $\cU = \mathrm{Sp}(\cO(\cU)) \subset \cW$. Points $\lambda \in \cU(L)$ correspond to maximal ideals $\m_\lambda$ of $\cO(\cU)$ with residue field $L$. We have a \emph{universal character} (or \emph{universal weight}) $\lambda_{\cU} : T(\Zp) \to \cO(\cU)^\times$ with the property that for any weight $\lambda \in \cU(L)$, the diagram
\begin{equation}\label{eqn:spec weights}
\begin{tikzcd}
	T(\Zp) \arrow[rrd, "\lambda"'] \arrow[rr, "\lambda_{\cU}"] && \cO(\cU)^\times \arrow[d,"\mathrm{sp}_\lambda"] \\
	&& L^\times
\end{tikzcd}
\end{equation}
commutes, where $\mathrm{sp}_\lambda$ is the natural evaluation map at $\lambda$, that is, reduction modulo $\m_\lambda$. We see that $\lambda_{\cU}$ is interpolating the single weights $\lambda$ over $\cU$.

\subsubsection{Distributions over weight space}
We now want to lift \eqref{eqn:spec weights} to the level of distributions, that is, interpolate the distributions modules $\cD_\lambda$ as $\lambda$ varies over $\cU$. More precisely, we want to define an $\cO(\cU)$-module of distributions $\cD_{\cU}$ over $\cU$ such that at any $\lambda \in \cU(L)$, we have
\[
\cD_{\cU}\otimes_{\lambda} L \defeq \cD_{\cU} \otimes_{\cO(\cU)}\cO(\cU)/\m_\lambda \cong \cD_\lambda(L).
\]
We can do this in a number of equivalent ways. The closest to the above is to consider the character $\lambda_{\cU}$ acting on a free $\cO(\cU)$-module of rank 1, and then define
\[
\cA_{\cU} = \mathrm{LA}\cdot\mathrm{Ind}_{B^-(\Zp)\cap I_p}^{I_p} \lambda_{\cU}.
\]
This is an $\cO(\cU)$-module which specialises to $\cA_{\lambda}$ upon reduction modulo $\m_\lambda$ by definition. The module $\cD_{\cU}$ satisfying the required property is essentially the continuous $\cO(\cU)$-dual of $\cA_{\cU}$ (see \cite[\S3]{Urb11} or \cite[\S2]{Han17}).

\subsubsection{Slope decompositions}
Again, for an open compact subgroup $K \subset G(\A_f)$ with $K_p \subset I_p$, the action of $I_p$ gives a local system $\sD_{\cU}$ on $Y_K$. The cohomology groups $\hc{i}(Y_K,\sD_{\cU})$ will not in general be finitely generated over $\cO(\cU)$. To use commutative algebra techniques, we wish to cut them down by taking \emph{slope decompositions}. Let $U_p$ be a controlling operator; since this acts compactly on the cohomology, after possibly shrinking $\cU$, for any $h \in \Q_{\geq 0}$ there exists a \emph{slope $\leq h$ decomposition} 
\[
\hc{i}(Y_K,\sD_{\cU}) = \hc{i}(Y_K,\sD_{\cU})^{\leq h} \oplus \hc{i}(Y_K,\sD_{\cU})^{> h}
\] 
of the cohomology, where $\hc{i}(Y_K,\sD_{\cU})^{\leq h}$ is a finitely generated $\cO(\cU)$-module. This is the subspace upon which $U_p$ acts with eigenvalues $\alpha$ all of slope $v_p(\alpha)\leq h$. In particular, if $h \geq v_p(\alpha_p)$, where $\alpha_p$ is the $U_p$ eigenvalue of our class $\phi_{\refpi}$, then all of the interesting cohomology at $\refpi$ is contained in the slope $\leq h$ cohomology. Importantly, the slope $\leq h$ decomposition is preserved by the Hecke operators. For further details, see \cite{AS08}.

For the rest of the paper, we will always assume $\cU$ is small enough to define a slope $\leq h$ decomposition for some $h$ large relative to the eigenvalues of $U_p$.

\subsubsection{Hecke algebras}
Let $\bH_{K}$ denote the abstract Hecke algebra of level $K$. This is the free $\Z$-algebra generated by an appropriate set of Hecke operators, which typically includes all Hecke operators $T_v$ at `good' primes not dividing the level, diamond operators at bad primes not dividing $p$, and a collection of $U_p$ operators at $p$. These Hecke operators are defined using correspondences on the cohomology. For precise details in the general setting, see \cite[\S1.1]{Han17}. Systems of Hecke eigenvalues correspond to non-trivial characters $\bH_K \to \overline{\Q}_p$, that is, to maximal ideals $\m \subset \bH_K\otimes_{\Z} \overline{\Q}_p$. 

\begin{example} \label{ex:f alpha}
	For $\GL_2/\Q$, we have that $\bH_{\Gamma_0(Np)}$ is the free $\Z$-algebra generated by $T_q$ for $q\nmid Np$, $\langle u\rangle$ for $u \in (\Z/N)^\times$ and $U_p$. If $f(z) = \sum_{n\geq 1}a_nq^n \in S_k(\Gamma_0(N))$ is a normalised newform with $p\nmid N$, let $\alpha \neq \beta$ be the (assumed distinct) roots of the Hecke polynomial $X^2 - a_pX + p^{k-1}$. Let $\Q(f_\alpha)$ be the number field generated by the $a_n$ and $\alpha$. The $p$-stabilisation of $f$ at $\alpha$ is $f_\alpha \defeq f(z) - \beta f(pz) \in S_k(\Gamma_0(Np))$. This is an eigenform, and we obtain a maximal ideal $\m_{f_{\alpha}} \subset \bH_{\Gamma_0(Np)}\otimes_{\Z}\Q(f_\alpha)$ generated by $T_q - a_q$ for $q\nmid Np$, $\langle u\rangle - 1$ for $u|N$, and $U_p - \alpha$. 
\end{example}

Let $R$ be a commutative ring, and let $\bH_{K,R} \defeq \bH_K \otimes_{\Z}R$. If $M$ is an $R$-module upon which $\bH_K$ acts $R$-linearly, let
\[
\bT_{K}(M) \defeq \text{ Image of }\bH_{K,R}\text{ in End}_{R}\big(M).
\]
If $\m \subset \bH_{K,R}$ is a maximal ideal, we say that $\m$ \emph{occurs in $M$} if $\bH_{K,R}/\m$ has non-trivial image in $\bT_K(M)$; that is, $\m$ induces a maximal ideal in $\bT_K(M)$, and the corresponding system of Hecke eigenvalues appears in $M$. 

\begin{remark}\label{rem:loc}
	An equivalent condition for $\m$ occurring in $M$ is that the algebraic localisation 
	\[
	M_{\m} = M \otimes_{\bH_{K,R}} (\bH_{K,R})_{\m}
	\]
	of $M$ at $\m$ is non-zero, whence it is isomorphic to $M \otimes_{\bT_K(M)}\bT_K(M)_{\m}.$ 
\end{remark}

\begin{example}
	We have actions of the Hecke operators on $\hc{i}(Y_K,\sM)^{\leq h}$ for $\sM = \sV_\lambda^\vee, \sD_\lambda, \sD_{\cU}$. Let  $f$ and $f_\alpha$ be as in the example above, and let $h \geq v_p(\alpha)$. We know that $\m_{f_\alpha}$ (always) occurs in $\hc{1}(Y_K,\sV_\lambda^\vee)^{\leq h}$, through the class $\phi_{f_\alpha}$, which is a Hecke eigenclass with the same eigenvalues as $f_\alpha$. Via Theorem \ref{thm:control} and the example following it, we see that if $v_p(\alpha) < k+1$, it occurs in $\hc{1}(Y_K,\sD_\lambda)^{\leq h}$ through the class $\Phi_{f_{\alpha}}$ lifting $\phi_{f_\alpha}$. (In fact, thanks to work of Bella\"iche, this is also known for $v_p(\alpha) = k+1$, hence for \emph{all} finite slope modular forms). To study the variation of this in families, we want to show that $\m_{f_\alpha}$ also appears in $\hc{1}(Y_K,\sD_{\cU})^{\leq h}$, which we examine in the sequel.
\end{example}

\subsubsection{Local geometry of the eigenvariety}\label{sec:local geometry}
Let $\cU \subset \cW$ be an affinoid subspace, fix the level group $K$ with $K_p \subset I_p$, and let 
\[
\bT_{\cU,h}^i \defeq \bT_{K}\big(\hc{i}(Y_K,\sD_{\cU})^{\leq h}\big)
\]
(where we take $R = \cO(\cU)$). The \emph{local piece of the eigenvariety} (over $\cU,$ of slope $\leq h$) is
\[
\cE_{\cU,h}^i \defeq \mathrm{Sp}(\bT_{\cU,h}^i).
\]
This comes equipped with a natural weight map 
\[
\cE_{\cU,h}^i \longrightarrow \cU
\]
given by the $\cO(\cU)$-structure map on the Hecke algebra. By definition, the $L$-points of $\cE_{\cU,h}^i$ are maximal ideals of $\bT_{\cU,h}^i$, which by the above discussions correspond to systems of Hecke eigenvalues appearing in $\hc{i}(Y_K,\sD_{\cU})^{\leq h}.$ (Whilst not relevant to our applications, we point out that -- as explained in \cite{Han17} -- via the `eigenvariety machine' of \cite{Buz07} these local pieces can be glued together into `global' eigenvarieties, hence the terminology.)

We now return to our fixed automorphic representation $\pi$, with a fixed choice of $p$-refinement $\refpi$ (and corresponding eigenform $f \in \pi$). We fix $K = K(\refpi)$. The Hecke eigenvalues of $f$ give rise to a maximal ideal $\m_{\refpi} \subset \bH_{K,\overline{\Q}_p}$, and it is natural to consider the localised cohomology groups $\hc{i}(Y_K,\sM)_{\m_{\refpi}}$, with $\sM = \sV_{\lambda}^\vee, \sD_\lambda, \sD_{\cU}$ in line with Remark \ref{rem:loc}. To simplify notation, for $M$ a module with an action of the Hecke algebra, we will write
\[
M_{\refpi} \defeq M_{\m_{\refpi}}.
\]
Motivated by Theorem \ref{thm:control}, we are led to the following definition.

\begin{definition}\label{def:non-critical}
	We say $\refpi$ is \emph{non-critical} if the localisation
	\[
	\rho_{\m_{\refpi}} : \hc{i}(Y_K,\sD_{\lambda})_{\refpi} \isorightarrow \hc{i}(Y_K,\sV_{\lambda}^\vee)_{\refpi}
	\]
	is an isomorphism for each $i$.
\end{definition}
In particular, since this is a further localisation of the spaces in Theorem \ref{thm:control}, we see that if $v_p(\alpha) < C_\lambda(t)$ (as in the notation of the theorem), then $\refpi$ is non-critical. 

This definition is extremely useful when studying families. In particular, at non-critical $\refpi$ it gives a complete description of the spaces $\hc{i}(Y_K,\sD_\lambda)_{\refpi}$, from which we can deduce properties of the $\hc{i}(Y_K,\sD_{\cU})_{\refpi}$, and we have the following

\begin{theorem}[Urban, Hansen]
	\label{thm:total cohom}
	Suppose $\refpi$ is non-critical, and $\cU$ a neighbourhood of $\lambda$ in $\cW$. Then there exists an $i$ such that $\m_{\refpi}$ occurs in $\hc{i}(Y_K,\sD_{\cU})$. 
\end{theorem}
\begin{proof}
	Urban treats the case of discrete series in \cite{Urb11}. In general, as $\tilde\pi$ is non-critical, $\m_{\tilde\pi}$ occurs in the total cohomology $\hc{\bullet}(Y_K,\sD_\lambda)$; then by \cite[Thm.\ 4.3.3]{Han17} this yields a point in his (total cohomology) eigenvariety, so that $\m_{\tilde\pi}$ occurs in $\hc{\bullet}(Y_K,\sD_{\cU})$. Since Hecke operators preserve degree of cohomology, it must appear in degree $i$ for some $i$.
\end{proof}

In particular, if we define $\bT^\bullet_{\cU,h} \defeq \bT_K(\hc{\bullet}(Y_K,\sD_{\cU})^{\leq h})$ using the \emph{total} cohomology, then the resulting eigenvariety $\cE_{\cU,h}^\bullet$ contains points for all non-critical $\refpi$ of slope $\leq h$ at $p$.

\subsection{The non-abelian Leopoldt conjecture}
\label{sec:leopoldt}
We pause briefly to describe a conjecture on the dimension of irreducible components, which gives significant control on the overconvergent cohomology in families. Define $t = q+\ell$ to be the top degree in which $\tilde{\pi}$ appears in classical cohomology. The following was conjectured by Hida and Urban (see \cite[Conj.~5.7.3]{Urb11}, \cite[Conj.~1.1.4]{Han17}):

\begin{conjecture}\label{conj:dim}
	Let $x_{\tilde\pi} \in \cE_{\cU,h}^\bullet$ be the point attached to a non-critical cuspidal $\tilde\pi$ in Theorem \ref{thm:total cohom}. Any irreducible component $\cV$ of the eigenvariety passing through the cuspidal point $x_{\refpi}$ has dimension $\mathrm{dim}(\cW) - \ell$.
\end{conjecture}

\begin{remark}In \cite[\S 5.7.3]{Urb11} it is observed that when $G= \mathrm{Res}_{E/\Q} \mathrm{SL}_2$ for $E$ a totally real number field, this conjecture is equivalent to the Leopoldt conjecture, from where the conjecture derives its name. It has been proved if $\ell \leq 1$, and more generally the inequality $\mathrm{dim}(\cV) \geq \mathrm{dim}(\cW)-\ell$ is proved in \cite[Prop.~B.1]{Han17}.
\end{remark}

A consequence of this conjecture is that when $\ell = 0$ -- that is, when $G(\R)$ admits discrete series -- all cuspidal families contain a Zariski-dense set of classical points. In this case the constructions of \cite{Urb11} give a complete answer to this question. Note, however, that the same is not true when $\ell \geq 1$.

\begin{definition}
	We say an algebraic weight $\lambda \in \cW$ is \emph{pure} if it supports cuspidal classical cohomology (in the sense that $\mathrm{H}_{\mathrm{cusp}}^\bullet(Y_K,\sV_{\lambda}^\vee) \neq 0$ for sufficiently small $K$). We write $\cW^{\mathrm{pure}}$ for the Zariski-closure of the pure weights, and call this the \emph{pure weight space}.
\end{definition}
We expect $\cW^{\mathrm{pure}}$ to have codimension $\ell$ in $\cW$. If $G = \GL_n$, then an algebraic weight $\lambda = (\lambda_1,...,\lambda_n)$ is pure if and only if $\lambda_i + \lambda_{n+1-i}$ is independent of $i$ \cite{Clo90}.

If $\cV$ is an irreducible component through a cuspidal point $x_{\refpi}$, we do not in general have control over the image of $\cV$ in $\cW$: this image may intersect the pure weight space only at the weight of $\refpi$, which means that $\cV$ cannot contain any other classical cuspidal points. In this case, interpolation (at a single point) is an empty statement. 

If $\cV$ is a component whose image in weight space is the pure weight space, then necessarily $\cV$ contains a Zariski-dense set of classical points. Indeed, there are a Zariski-dense set of classical weights in $\cW^{\mathrm{pure}}$, and since the slope at $p$ is constant in families, the family will have small slope at a Zariski-dense set of these. The systems of Hecke eigenvalues parametrised in the overconvergent cohomology at these points will thus be classical. We generally expect that such components $\cV$ only arise as a Langlands transfer from a group that does admit discrete series, for example as in the $p$-adic base-change map of \S\ref{sec:p-adic base-change}. This was observed for $\GL_3$ in \cite{APS08}, and more generally in the introduction of \cite{Urb11}. In this sense, the existence of a Zariski-dense set of classical points would appear to only exist for `degenerate' reasons.

\begin{example}
	Suppose $E$ is imaginary quadratic, and $G = \mathrm{Res}_{E/\Q}\GL_2$ (so we are in the setting of Bianchi modular forms). Bianchi cusp forms exist only at parallel weights $(k,k)$ \cite{Har87}, so the pure weight space is a line in the two-dimensional weight space. It is precisely the set of characters $(\cO_E\otimes_{\Z}\Zp)^\times \to \Cp^\times$ that factor through the norm to $\Zp^\times$. In particular, it is the image of the weight space $\cW_{\GL_2} \hookrightarrow \cW_G$ under the natural embedding. So-called `genuine' Bianchi modular forms, those that do not come from base-change from $\Q$, are relatively rare \cite{RS13,RT19}. This is reflected in a conjecture of Calegari--Mazur \cite{CM09}, who predict that these genuine modular forms are so scarce that they cannot exist in classical $p$-adic families: more precisely, that the only classical families of Bianchi modular forms are CM or come from base-change.
\end{example}

\begin{remark} 
	For Bianchi modular forms, a concrete proof of the conjecture for arbitrary cuspidal families (provided by David Hansen) was given in \cite[Theorem 3.8]{BW18}, using the fact that a 2-dimensional family would necessarily contain a Zariski-dense set of classical non-critical Eisenstein series, since the affinoid in weight space has codimension 0 and thus contains a Zariski-dense set of classical weights. This forces reducibility of the Galois pseudo-character, and thus reducibility of the Galois representation of the Bianchi cusp form, which is a contradiction. Note that the same argument does not work when $\ell > 1$; any affinoid of positive codimension in $\cW$ might not contain \emph{any} classical weights outside $\lambda$.
\end{remark}

\subsection{Top row, II: families of $p$-adic $L$-functions}\label{sec:top row II}
The above formalism admits clear connections to the constructions of $p$-adic $L$-functions in the middle row. There, we wanted to interpolate the the bottom row as the parameter $\varphi$ varied. We argued in three steps: 
\begin{enumerate}[(1)]\setlength{\itemsep}{0pt}
	\item we first considered a much larger coefficient sheaf $\sD_\lambda$;
	\item we then exhibited a distinguished class $\Phi_{\refpi} \in \hc{d}(Y_K,\sD_\lambda)$;
	\item and finally we constructed a Galois evaluation map $\mathrm{Ev}_\lambda$, built using the maps of \ref{sec:abstract evaluation} and valued in $\cD(\Galp,\overline{\Q}_p)$, interpolating all the $\mathrm{Ev}_{\varphi}$. 
\end{enumerate}
To construct a family of $p$-adic $L$-functions, we instead want to interpolate the \emph{middle} row as the parameter $\lambda$ varies. We argue in a very similar fashion. Suppose $\refpi$ gives rise to a point $x_{\refpi}$ in the eigenvariety $\cE_{\cU,h}^\bullet$, and let $\cV$ be a neighbourhood of $x_{\refpi}$ in $\cE_{\cU,h}^\bullet$ lying over a Zariski-closed $\Sigma \subset \cU$ containing $\lambda$. Then we use the following strategy:
\begin{enumerate}[(1)]\setlength{\itemsep}{0pt}
	\item\label{sheaf} We pass to a bigger coefficient sheaf, $\sD_{\Sigma}$ as above, interpolating the $\sD_\lambda$.
	
	\item\label{class} Then we exhibit a canonical class $\Phi_{\cV} \in \hc{d}(Y_K,\sD_{\Sigma})$ interpolating the $\Phi_{\refpi_y}$ as $y$ varies over classical points in $\cV$ corresponding to refined automorphic representations $\refpi_y$.
	\item\label{big ev} Finally, we build a compatible system of evaluation maps $\mathrm{Ev}^c_{\cV}$, valued in $\cD(\Galp,\cO(\cV))$, interpolating the $\mathrm{Ev}^c_{\lambda_y}$ in the sense that
	\begin{equation}\label{eqn:ev int}
		\mathrm{Ev}^c_{\lambda_y} \circ \mathrm{sp}_{\lambda_y} = \mathrm{sp}_{\lambda_y} \circ \mathrm{Ev}^c_{\cV},
	\end{equation}
	where $\lambda_y$ is the weight of $\refpi_y$, and $\mathrm{sp}_{\lambda_y}$ is induced from $\cO(\Sigma) \to \cO(\Sigma)/\m_{\lambda_y} \subset \overline{\Q}_p$. This forms the commutative square of rows [T] and [M] in Figure \ref{main diagram}.
\end{enumerate}

We have already described \eqref{sheaf}. Part \eqref{big ev} is built in exactly the same manner as \S\ref{sec:galois evaluations}; by simply replacing the coefficients and arguing almost identically, one obtains a map $\mathrm{Ev}^c_{\Sigma}$ into $\cD(\Gal_p(E),\cO(\Sigma))$, and $\mathrm{Ev}^c_{\cV}$ is obtained by tensoring with $\cO(\cV)$ over $\cO(\Sigma)$. That this map satisfies the property \eqref{eqn:ev int} follows using the covariance of Remark \ref{rem:covariance} applied to the map $\cD_{\Sigma} \to \cD_\lambda$ of $K$-modules.

It remains to treat \eqref{class}. This is much more subtle. A particular subtlety arises as the construction of $p$-adic $L$-functions typically works using a specific degree of cohomology $d$, and we must thus also exhibit our class in \emph{this} degree. Working over full dimensional affinoids, however, we expect the overconvergent cohomology to be supported only in the \emph{top} degree (see Newton's remark at the end of \cite[App.~B]{Han17}). To access lower degrees, we must study a more constrained eigenvariety; namely, we aim to show that $\cV$ is an irreducible component of the eigenvariety $\cE_{\Sigma,h}^d$, where we constrain the weight space to a smaller space $\Sigma \subset \cU$ and the degree to $d$ instead of the total cohomology. The existence of $\Phi_{\cV}$ can then often be shown by studying the geometry of $\cV$ in $\cE_{\Sigma,h}^d$.

In the second part of this paper, we restrict to the case of $\GL_2$ over a CM field, and give a construction of $\Phi_{\cV}$ in one-dimensional families under the additional hypothesis of Conjecture \ref{conj:dim}. This conjecture provides the additional control required to exhibit classes in the correct degree (Corollary \ref{cor:non-vanishing}), and its necessity in our methods reflects the subtlety of working in situations where cohomology appears in multiple degrees.

\subsection{Known cases}
We end this survey by briefly indicating some cases in which (parts of) this program have been successfully carried out.
\begin{enumerate}[--]\setlength{\itemsep}{0pt}
	\item \textbf{Classical modular forms}. The first construction of this kind was given for Hida families of (ordinary) modular forms in \cite{GS93}.  Closer to the strategy described in this survey, and the main inspiration of it, is the original construction of Stevens in \cite{Ste94}, later explored by Pollack and Stevens in \cite{PS11}. This gives the construction for small slope forms -- that is, $f$ with $v_p(a_p(f)) < k+1$ -- using classes in $\hc{1}$ which, in this case, can be described more concretely via \emph{modular symbols}. They extended the construction to certain critical slope forms (with $v_p(a_p(f)) = k+1$) in \cite{PS12} by giving a finer criterion for the form to be non-critical (in the sense of Definition \ref{def:non-critical}).
	
	The case of critical $f$ was settled in \cite{Bel12}, who used the geometry of the eigencurve to provide the canonical class $\Phi_{\refpi}$ in this case. We note that for such $f$, this class is mapped to zero in the classical cohomology, and hence the interpolation formula is always zero at every special $L$-value. Bella\"iche also varied this construction in families.
	
	\item \textbf{Hilbert modular forms}. Let $E$ be totally real of degree $d$. The first construction for single, small slope forms for $G= \mathrm{Res}_{E/\Q}\GL_2$ appeared in \cite{Bar15}, using $\hc{d}$ and the automorphic cycles of \cite{Dim13}. This was generalised to critical forms in \cite{BH17}, and variation in families was given in \cite{BH17,BDJ17}.
	
	\item \textbf{Bianchi modular forms}. 
	Let $E$ be an imaginary quadratic field. \cite{Wil17} contained a construction for single, small slope forms for $G= \mathrm{Res}_{E/\Q}\GL_2$, using $\hc{1}$ and its description via modular symbols. This was generalised to critical base-change forms in \cite{BW18}, which also gave interpolation in base-change families. We note that this is the first setting where forms appear in multiple degrees of cohomology, so forms that are not base-change or CM are not expected to vary in classical families.
	
	\item $\mathbf{GL}_2$ \textbf{over a general number field $E$}. In this case, a construction was given in \cite{BW_CJM}, using $\hc{q}$, where $q = r+s$ is the bottom degree in which automorphic representations appear (where $E$ has $r$ real embeddings and $s$ pairs of complex embeddings).
	
	\item $\mathbf{GL}_{2n}$\textbf{ over totally real fields}. In this case, a construction for small slope \emph{symplectic} automorphic representations -- including variation in families -- is the main result of \cite{BDW20}.
	
	\item \textbf{Square root $p$-adic $L$-functions}. In a different direction, there are numerous constructions of $p$-adic variation in families of a \emph{single} value: that is, variation in the weight without variation in the cylotomic variable. An example of this sort is found in \cite{BD09}, where overconvergent cohomology is used to construct a $p$-adic measure on a Hida family $\mathcal{R}$ whose square interpolates the central critical values $L(f/K, \chi_K, k)$ at classical $f \in \mathcal{R}$ of even weight. Here $K$ is a real quadratic field and $\chi_K$ its genus character.
	
	\item \textbf{Adjoint $p$-adic $L$-functions}. On a similar theme, the main result of \cite{Lee21} is a construction, via overconvergent $\hc{2}$, of a $p$-adic $L$-function over a Hida family $\mathcal{R}$ that interpolates (non-critical) values of the adjoint $L$-function attached to $\mathcal{R}$ and an imaginary quadratic field. Again, there is no cyclotomic variation in this construction.
	
	\item \textbf{Triple product $p$-adic $L$-functions}. In \cite{GS20}, the authors use overconvergent methods to construct triple product $p$-adic $L$-functions in the balanced region of weight space, interpolating the central critical value of the triple product $L$-function of a trio of Coleman families over this region.
	
\end{enumerate}


\section{Families of $p$-adic $L$-functions over CM fields}\label{sec:CM families}

In the remainder of the paper, we show how to carry out the last part of the strategy described above in a new situation: producing families of $p$-adic $L$-functions for $\GL_2$ over CM fields. In this case, the construction for a single form (valid over arbitrary number fields) was given in our earlier paper \cite{BW_CJM}, to which we refer the reader for further details on our conventions.

\subsection{Preliminaries and notation}\label{sec:preliminaries}
Let now $E$ be a CM field, with totally real subfield $E^+$, and fix $G = \mathrm{Res}_{E/\Q}\GL_2$ and $G^+ \defeq \mathrm{Res}_{E^+/\Q}\GL_2$. Let $d = [E^+:\Q]$, let $\cO_E$ (resp.\ $\cO_{E^+}$) denote the ring of integers and let $\Sigma_E$ (resp.\ $\Sigma_{E^+}$) denote the set of complex embeddings of $E$ (resp.\ $E^+$). We maintain all of the general notation of \S\ref{sec:survey}, with the convention that a superscript $+$ always denotes an object for $E^+$. 
For simplicity, we write $\mathrm{Gal}_p^+= \mathrm{Gal}_p(E^+)$ and $\mathrm{Gal}_p= \mathrm{Gal}_p(E)$. In this setting, we have more concrete descriptions of the several objects involved, which we now highlight.

Let $\pi^+$ be an automorphic representation for $G^+(\A)$, which we will always assume does not have CM by $E$, and let $\pi$ be its base-change to $G(\A)$. To maintain consistency with \cite{BW_CJM}, rather than using the more conceptual language of automorphic representations, we instead switch to the language of automorphic forms. In particular, throughout $f^+$ will be a cuspidal automorphic eigenform in $\pi^+$ -- a Hilbert modular form -- and $f$ its base-change, that is, a choice of automorphic eigenform in $\pi$ with the correct Hecke eigenvalues. We assume that $f^+$ and $f$ are new at every place $v\nmid p$. For $f^+$, we work with the concrete level group
\[
K_1(\n^+) \defeq \left\{\smallmatrd{a}{b}{c}{d} \in G^+(\A_f) : c \equiv 0 \newmod{\n^+}, d \equiv 1 \newmod{\n^+}\right\},
\]	
for $\n^+ \subset \cO_{E^+}$ an ideal. We assume for simplicity that $\n^+$ is coprime\footnote{This is not necessary, and is imposed only to simplify the notation; see the discussion in \cite[\S3]{BW18}.} to $\mathfrak{d} = \mathrm{disc}(E/E^+)$; then in particular, $f$ has level $K_1(\n)$, where $\n = \n^+\cO_E$ (and $K_1(\n) \subset G(\A_f)$ is the natural analogue of $K_1(\n^+)$). We assume that $\n^+$ (hence $\n$) is divisible by all the primes of $E^+$ (resp.\ $E$) above $p$, so that $f$ corresponds already to the choice of $p$-refined eigenform in the data of $\refpi$. We write $Y_1(\n) = Y_{K_1(\n)}$.

For $\pri|p$ in $E$, we take $U_{\pri}$ to be the usual Hecke operator corresponding to the double coset operator $[K_1(\n)\smallmatrd{1}{}{}{\varpi_\pri}K_1(\n)]$, where $\varpi_\pri$ is a uniformiser at $\pri$. Let $\alpha_{\pri}$ be the $U_{\pri}$-eigenvalue of $f$, which we assume, for all $\pri|p$, to be a simple eigenvalue of $U_{\pri}$ on $\pi^{K_1(\n)}$, so that either $f$ is new at $\pri$, or $\pi$ is unramified at $\pri$, $K_1(\n)_{\pri} = I_{\pri}$ and $U_p$ has distinct eigenvalues on $\pi_{\pri}^{I_{\pri}}$, where $I_{\pri}$ is the Iwahori (similarly to \S\ref{sec:set-up}). 

The algebraic representation $V_\lambda$ is a tensor product of polynomial rings (see \cite[\S2.1]{BW_CJM}). We consider the cohomology groups $\hc{i}(Y_1(\n),\sV_\lambda^\vee(L))$ exactly as in \S\ref{sec:survey}. In this setting, the cohomology at $f$ is supported in degrees $i \in [d,d+1,...,2d]$, and we work in the bottom degree cohomology. Our assumption that $\alpha_{\pri}$ is a simple eigenvalue at each $\pri$, and that $f$ is new at all other primes, ensures that for sufficiently large $L/\Qp$,
\begin{equation}\label{eqn:classical line}
	\mathrm{dim}_{L}\hc{d}(Y_1(\n),\sV_\lambda^\vee(L))_{f} = 1,
\end{equation}
where we write $\m_f$ and a subscript $f$ in place of $\m_{\refpi}$ and a subscript $\refpi$ from \S\ref{sec:survey}. Implicit in this statement is the choice of the complex period $\Omega_{f} \in \C^\times$ and rationality of the cohomology from \cite{Har87,Hid94}. We denote by $\phi_f \in \hc{d}(Y_1(\n),\sV_\lambda^\vee(L))$ the resulting class, which is canonical up to scaling by an element of $\Q(f)^\times$, the number field generated by the Hecke eigenvalues of $f$. This fact is important when we vary in families.

Let $\Lambda(f,\varphi)$ denote the (completed) $L$-function of $f$, normalised as in \cite[\S3]{BW_CJM}.


\subsection{Abstract evaluation maps}\label{ss: CM abstract evaluation maps}

We introduce the \emph{evaluation maps} considered in the context of $ \mathrm{Res}_{E/\Q}\GL_2$. Such evaluations were essentially introduced in \cite[\S5, \S10]{BW_CJM} and we adjust the constructions \emph{op.\ cit}.\ in order to simultaneously describe the bottom, middle and top rows of Figure \ref{main diagram}.

Firstly we recall the relevant automorphic cycles as in \S\ref{sec:abstract evaluation}. We denote $E^{1}_{\infty} \defeq \{z\in E_{\infty}: |z_v|_v= 1 \ \text{for all} \ v\mid \infty\}\cong (\mathbb{S}^1(\R))^d$.  For an ideal $\fc$ of $E$ we denote $U(\fc)\defeq \{ u\in \hat{\cO}_K^{\times}: u-1 \in \fc\hat{\cO}_K \}$ and we define 
\[
X_{\fc}\defeq E^{\times}\backslash \bA_{E}^{\times}/U(\fc)E^{1}_{\infty}.
\]
In the language of \S\ref{sec:abstract evaluation} we consider $H= \mathrm{Res}_{E/\Q}\mathrm{GL}_1$, $G= \mathrm{Res}_{E/\Q} \mathrm{GL}_2$ and $\iota: H \hookrightarrow G$ given by the morphism $\mathrm{GL}_1 \hookrightarrow \mathrm{GL}_2, a \mapsto \smallmatrd{a}{0}{0}{1}$ . 

The set of connected components of $X_{\fc}$ is naturally described as $\pi_{0}(X_{\fc})\cong \mathrm{Cl}_{E}(\fc)$, noting that this is also the narrow ray class group, as $E$ has no real embeddings. In fact, we have the following decomposition of $X_{\fc}$ into connected components:
\begin{align*}
	X_{\fc}= \bigsqcup_{a \in \pi_{0}(X_{\fc})} X_{\fc}^{a}&\cong \left[\cO_{E, \fc}^{\times}\backslash (E_\infty^{\times}/ E_{\infty}^1)\right]^{\pi_{0}(X_{\fc})}\\
	&\cong \left[\cO_{E, \fc}^{\times}\backslash (\C^\times/\mathbb{S}^1(\R))^d\right]^{\pi_{0}(X_{\fc})} \cong \left[\cO_{E, \fc}^{\times}\backslash \R^d_{+} \right]^{\pi_{0}(X_{\fc})}.
\end{align*}
In the notation of \S\ref{sec:abstract evaluation}, we have $\Gamma_{\fc, a}= \cO_{E, \fc}^{\times}$ for each $a \in \pi_{0}(X_{\fc})$, and in particular $\Gamma_{\fc, a}$ is independent of $a$.

Write $\fc= \prod_{v}v^{n_v}$  and let $\varpi_{\fc}^{-1}\in \A_E$ and $i_{\fc}\in \A_E^{\times}$ be defined as follows:  if $n_v\neq 0$ we put $(\varpi_{\fc}^{-1})_v= \varpi_v^{-n_v}$ and $(i_{\fc})_v= \varpi_v^{n_v}$; if $n_v= 0$ then we put $(\varpi_{\fc}^{-1})_v=0$ and $(i_{\fc})_v= 1$  (for $\varpi_v$ fixed uniformisers in $\cO_{E_v}$). The \emph{automorphic cycle} of level $\fc$ is the embedding
\[
\iota_{\fc}: X_{\fc} \hookrightarrow Y_1(\fn), \ \ [x] \mapsto \left[\smallmatrd{x}{x\varpi_{\fc}^{-1}}{0}{1}\right].
\]
(This was denoted $\eta_{\fc}$ in \cite{BW_CJM}). When $\fc = 1$, we write $\iota\defeq \iota_1$. In order to apply the formalism described in \S\ref{sec:abstract evaluation} to obtain general evaluation maps, we need to make precise the ``appropriate twists'' of Construction \ref{cons:eval}, step 2.  Let $M$ be a module admitting a right action of a semigroup containing $K_1(\fn)$ and $\smallmatrd{1}{-1}{0}{i_{\fc}}$. Then we consider the twist
\[
\zeta_* : \hc{d}(X_\fc, \iota_\fc^{\ast}\sM) \rightarrow \hc{d}(X_\fc, \iota^{\ast}\sM)
\]
induced by the map 
\begin{align*}
	\zeta : \mathbb{A}_E^{\times}\times M &\longrightarrow \mathbb{A}_E^{\times}\times M,\\ 
	(x, m)&\longmapsto \left(x, m \left|\smallmatrd{1}{-1}{0}{i_{\fc}}\right.\right).
\end{align*}

Philosophically, we twisted the automorphic cycle to ensure the connected components were mapped into different parts of $Y_1(\n)$, but this also caused the local system to become twisted; $\zeta$ `untwists' the local system. Then $\widetilde{\iota}_{\fc} \defeq \zeta_* \circ \iota_{\fc,*}$ is the map required in Construction \ref{cons:eval}.

We fix $S_{\fc}$ an idelic system of representatives of $\mathrm{Cl}_E(\fc)$, and for each $a \in \pi_{0}(X_{\fc})$, we denote by $x_a\in S_{\fc}$ the corresponding idele. The procedure of \S\ref{sec:abstract evaluation} gives, for each $a \in \pi_{0}(X_{\fc})$, an evaluation map
\[
\mathrm{Ev}_{\fc, x_a}^M: \hc{d}(Y_1(\fn), \sM) \rightarrow M_{\cO_{E, \fc}^{\times}}.
\]
These are naturally compatible in changing $x_a$: if $x_a' = \gamma x_a ur$ is a different choice, with $u \in U(\fc)$, then $\mathrm{Ev}_{\fc,x_a'}^M |\smallmatrd{u}{}{}{1} = \mathrm{Ev}_{\fc,x_a}$ (see the proof of \cite[Prop.~10.6]{BW_CJM}; the proof works for arbitrary $M$). As outlined in \S\ref{sec:abstract evaluation}, they are naturally covariant in $M$.

\subsection{Bottom and middle rows}

Firstly we specialise the general discussion in \ref{sec:bottom row} about the bottom row, and give a description of the critical $L$-values of $f$ in terms of the evaluations of \S\ref{ss: CM abstract evaluation maps}. We follow \cite{BW_CJM}.

Recall from \S\ref{sec:preliminaries} the modular form $f$ of weight $\lambda= (k, v)\in \Z[\Sigma_E]\times \Z[\Sigma_E]$ and the class $\phi_{f} \in \hc{d}(Y_1(\fn), \sV_\lambda^{\vee}(L))$, which in fact is defined over the number field $\Q(f)$. Recall $V_\lambda$ is a space of (multivariable) polynomials. Let $j\in \Z[\Sigma_E]$ and consider the map $\rho_j: V_\lambda^\vee(\Q(f)) \rightarrow \Q(f)$ introduced in \cite[\S 5.2.3]{BW_CJM}; this is given by evaluation at some precise monomial of exponent $j+v$. If $j + v$ is the infinity type of a Hecke character, then it is shown \emph{op.\ cit}.\ that this map naturally descends to the coinvariants $(V_\lambda^\vee(\Q(f)))_{\cO_{E,\fc}^\times}$. Choosing representatives $x_a$ of $\mathrm{Cl}_E(\fc)$, and composing, we obtain an evaluation
\[
\mathrm{Ev}_{\fc, x_a, j}\defeq \rho_j \circ \mathrm{Ev}_{\fc,x_a}^{V_{\lambda}^\vee(\Q(f))} : \hc{d}(Y_1(\fn), \sV_\lambda^\vee(\Q(f)))\longrightarrow \Q(f).
\]
The following is \cite[Thm.~5.5]{BW_CJM}, and completes the bottom row calculations:
\begin{proposition} \label{p: CM L-values and evaluations}  For each Hecke character $\varphi$ of $E$ of conductor $\fc$ and infinity type $j+ v$ such that $0 \leq j \leq k$ we have
	\[
	\mathrm{Ev}_{\varphi}(\phi_f) \defeq \sum_{a \in \mathrm{Cl}_E(\fc)}\varphi(x_a)\mathrm{Ev}_{\fc, x_a, j}(\phi_f)= (-1)^{r} \frac{|D|\tau(\varphi)}{2^{d}\Omega_f}\cdot\Lambda(f, \varphi),
	\]
	where $D$ is the discriminant of $E$, $\tau(\varphi)$ is the Gauss sum attached to $\varphi$ and $r= \sum_{\sigma \in \Sigma_{E}}k_{\sigma}$.
\end{proposition}

We now describe the middle row. Consider the overconvergent cohomology group $\hc{d}(Y_1(\fn), \sD_\lambda(L))$ for $\cD_\lambda(L)$ as defined in \S\ref{ss: survey overconvergent cohomology}. To obtain canonical lifts, we make precise the bound $C_\lambda(t)$ from Theorem \ref{thm:control} in this setting. We are in the stronger context of Remark \ref{rem:multi}(ii). 

\begin{definition}\label{def:small slope}
	Let $\mathbf{h}= (h_{\fp})_{\fp\mid p} \in \Q^{\fp\mid p}$. We say that $\mathbf{h}$ is a \emph{small} (or \emph{non-critical}) slope for the weight $\lambda = (k,v)$ if for each prime $\fp\mid p$ in $E$, we have 
	\begin{equation}\label{eqn:small slope}
		h_{\fp}< C_\lambda(\fp) \defeq \frac{1+ \mathrm{min}_{\sigma \in \Sigma_{\fp}}(k_{\sigma}) + \sum_{\sigma \in \Sigma_{\fp}}v_\fp}{ e_{\fp}}.
	\end{equation} 
	(Here $e_{\fp}$ is the ramification index of $\fp$, $\Sigma_{\pri}\subset \Sigma$ is the set of embeddings $E \hookrightarrow \overline{\Q}$ which extend to $E_{\pri}\hookrightarrow \overline{\Q}_p$. Note $\Sigma = \sqcup_{\pri}\Sigma_{\pri}$).
\end{definition}
The following control theorem is a precise version of Theorem \ref{thm:control} and was proved in \cite[Thm.~9.7]{BW_CJM}. For the notion of `multi-slope decompositions' with respect to multiple operators, a refinement of the theory explained in \S\ref{sec:survey}, see \cite[\S9]{BW_CJM}.

\begin{theorem}\label{thm:small-slope}
	Let $\mathbf{h}$ be a small slope for $\lambda$. Then after restricting to the slope $\leq \mathbf{h}$ subspaces of the $U_{\pri}$-operators, we have an isomorphism  
	$$\hc{d}(Y_1(\fn), \sD_\lambda(L))^{\leq \mathbf{h}} \cong \hc{d}(Y_1(\fn), \sV_\lambda^\vee(L))^{\leq \mathbf{h}}.$$
\end{theorem}

The control theorem applied to the class $\phi_{f} \in \hc{d}(Y_1(\fn), \sV_\lambda^\vee(L))^{\leq \mathbf{h}}$, which we assume to have small slope, yields a canonical class $\Phi_{f} \in \hc{d}(Y_1(\fn), \sD_\lambda(L))$. To construct the $p$-adic $L$-function from this class, following \S\ref{sec:galois evaluations}, for each ideal $\fc \mid (p^{\infty})$ we obtain a Galois evaluation map  
\[
\mathrm{Ev}^{\fc}_{\lambda}: \hc{d}(Y_1(\n), \sD_\lambda(L)) \rightarrow \cD(\mathrm{Gal}_p, L),
\]
constructed in \cite[\S10]{BW_CJM}. Here Expectation \ref{expectation} is shown in \cite[Prop.~10.9]{BW_CJM}.

\begin{definition} 
	The \emph{$p$-adic $L$-function attached to $f$} is $L_p(f)= \alpha_{p}^{-1}\mathrm{Ev}_{\lambda}^{(p)}(\Phi_{f})  \in \cD(\mathrm{Gal}_p, L)$, where $\alpha_{p}$ is the $U_p$-eigenvalue of $f$. 
\end{definition}

The following was the main result of \cite{BW_CJM} (Theorem 12.1). We maintain the notation of Proposition \ref{p: CM L-values and evaluations}.

\begin{theorem}\label{p:admissible} $L_p(f)$ is independent of all choices, and satisfies the following interpolation property: for all Hecke characters $\varphi$ of $E$ with infinity type $j+v$, for $0 \leq j \leq k$, and conductor $\fc|p^\infty$, we have
	\begin{align*}
		L_p(f,\varphi) &\defeq \int_{\Gal_p}\varphi_{(p)} \cdot dL_p(f)\\
		&= (-1)^r \left[\frac{|D|\tau(\varphi)\varpi_{\fc}^{j+v}}{2^d \alpha_{\fc}\Omega_f}\right] \left(\prod_{\fp|p,\ \fp\nmid \fc} Z_{\pri}\right)\cdot \Lambda(f,\varphi),
	\end{align*} 
	where $\varphi_{(p)}$ is the $p$-adic avatar of $\varphi$ as in \cite[Def.~2.3]{BW_CJM}, $\alpha_{\fc}$ is the $U_{\fc}$-eigenvalue of $f$ and $Z_{\fp} \defeq \varphi_{(p)}(\varpi_{\fp})(1-\alpha_{\fp}^{-1}\varphi(\fp)^{-1}).$
\end{theorem}


\subsection{Top row I: Existence of families}
We now prove the existence of $p$-adic families for $G$ arising as the base-change of those for $G^+$. 

\subsubsection{Weight spaces, revisited}
We can be more precise about the weight spaces involved. In \S\ref{sec:survey}, a weight for $G^+$ was a character on $T^+(\Zp)$. An algebraic weight then has the form $(k+v,v)$ for $k,v \in \Z[\Sigma_{E^+}]$, corresponding to the character
\[
\smallmatrd{z}{}{}{w} = \smallmatrd{(z_\sigma)_{\sigma}}{}{}{(w_\sigma)_{\sigma}} \longmapsto \prod_{\sigma} z_\sigma^{k_\sigma + v_\sigma} w_\sigma^{v_\sigma},
\]
where $\sigma$ runs over $\Sigma_{E^+}$. It is customary to work with a modified version of this; it is classically known \cite{Har87} that cuspidal automorphic representations for $G^+$ must have \emph{pure} weights, with $k+2v \in \Z[\Sigma_{E^+}]$ parallel (in the sense that $(k+2v)_\sigma= \sw$ for all $\sigma\in \Sigma_{E^+}$, for some $\sw \in \Z$ called the \emph{purity weight}). In particular, we may instead represent $(k+v,v)$ as $(k,\sw)$, since this entirely determines $v$. Since changing $\sw$ then corresponds to twisting by some power of the determinant, which corresponds only to changing the central character of $\pi$ by the norm, we may consider $\sw$ fixed and the weight to be determined by $k$. In particular, we may redefine $\cW^+$ to be the rigid space such its points are given by
\[
\cW^+(L) \defeq \mathrm{Hom}_{\mathrm{cts}}([\cO_{E^+}\otimes_{\Z}\Zp]^\times,L^\times),
\]
a $d$-dimensional space that captures all the interesting $p$-adic variation of weights. We similarly define $\cW^{\mathrm{full}}$ such that $\cW^{\mathrm{full}}(L) \defeq \mathrm{Hom}_{\mathrm{cts}}([\cO_{E}\otimes_{\Z}\Zp]^\times,L^\times)$, the full $2d$-dimensional weight space for $E$. We are primarily interested in a smaller subset: in particular, composing with the norm $\cO_E \to \cO_{E^+}$ gives a map $\cW^+ \hookrightarrow \cW^{\mathrm{full}}$ of rigid spaces, and we define
\[
\cW \defeq \mathrm{Image}(\cW^+) \subset \cW^{\mathrm{full}}.
\]
This is the space of `conjugate-invariant' weights i.e. the weights $\lambda = k \in \cW^{\mathrm{full}}$ with $k_\sigma = k_{c\sigma}$ for all $\sigma \in \Sigma[E]$, where $c$ is complex conjugation. After twisting $\pi^+$ (resp.\ $\pi$) by some power of the norm for $E^+$ (resp.\ $E$), we may consider it to have weight $\lambda^+ \in \cW^+$ (resp.\ $\lambda \in \cW$). Such a twist respects the property of being non-critical slope, since the $p$-adic valuations of the eigenvalues and the small-slope bound are translated by the same value.

\begin{remark}
	Again, there are level considerations here to ensure non-trivial local systems. However, we have already mitigated this by working in level $K_1$ and restricting to pure weights: if $\lambda^+ \in \cW^+$, then $\lambda(Z(G^+(\Q)) \cap K^+_1(\n)) = 1$.
\end{remark}

\subsubsection{The $p$-adic base-change map}\label{sec:base change transer}
\label{sec:p-adic base-change}
Let $f^+$ and $f$ be as in \S\ref{sec:preliminaries}, of weights $\lambda^+$ and $\lambda$, and suppose that $f$ is \emph{non-critical}. Our running assumptions, as explained in \S\ref{sec:preliminaries}, imply that $\mathrm{dim}_L \hc{d}(Y_1(\n),\sV_\lambda^\vee)_{f} = 1$. Let $\Omega^+ \subset \W^+$ be a ($d$-dimensional) affinoid neighbourhood of $\lambda^+$ in the Hilbert weight space, and let $\cE_{\Omega^+,h^+}^{\bullet,+}$ denote the local piece of the Hilbert eigenvariety over $\Omega^+$ (built from the total cohomology, as in \cite{Han17}). Let $\Omega$ denote the image of $\Omega^+$ in $\cW$, and let $\cU$ be a ($2d$-dimensional) affinoid neighbourhood of $\Omega$ in $\cW^{\mathrm{full}}$.

\begin{theorem}\label{t:transfert}
	For $h \geq 2h^+$, there is a finite map
	\[
	\BC : \cE^{\bullet, +}_{\Omega^{+}, h^+} \longrightarrow \cE^\bullet_{\cU, h}
	\]
	of $L$-rigid analytic spaces over $L$, interpolating base-change transfer on classical points.
\end{theorem} 
\begin{proof}
	This kind of $p$-adic Langlands functoriality was first explored by Chenevier; we use the very general formulation of \cite{JoNew}. The idea is to interpolate such a transfer on a Zariski-dense set of non-critical-slope classical points, which we have from the theory of automorphic base-change and Theorem \ref{thm:total cohom}. Firstly observe that both $\cE_{\Omega^+,h^+}^{\bullet,+}$ and $\cE_{\cU,h}^\bullet$ arise from \emph{eigenvariety data}, namely $(\Omega^+,\sZ^+,\sM^{\bullet,+}_{\Omega^+,h^+},\bH_{K^+}^+,\psi^+)$ and $(\cU,\sZ,\sM^\bullet_{\cU,h},\bH_K,\psi)$, for notation as in \cite[\S4.3]{Han17}. These are related in \cite[\S4.3]{JoNew}, formalising base-change transfer on the eigenvariety data; and then the map is constructed in \cite[Thm.~3.2.1]{JoNew}. We note that if a Hilbert modular form has slope $h^+$, then its base-change has slope $h \leq 2h^+$, and that any family will be non-critical slope at a Zariski-dense set of classical weights.
\end{proof}

We remark that the theory in \cite{JoNew} is stated in the language of adic spaces, but can be used in the language of rigid analytic spaces, which form a fully faithful subcategory of the category of adic spaces \cite{Hub94}.

\begin{proposition}\label{prop:pt}
	There is a point $x_{f^+} \in \cE^{\bullet,+}_{\Omega^+,h^+}$ corresponding to $f^+$. 
\end{proposition}
\begin{proof}
	This is \cite[Thm.~5.4.4]{Urb11} for $f^+$ non-critical. Since $f$ is non-critical, we expect $f^+$ should also be non-critical; but in any case, if it is critical, then the existence of $x_{f^+}$ can be deduced by equating the eigenvarieties from overconvergent modular forms and overconvergent cohomology again via a $p$-adic transfer theorem, noting that as the classical forms are a subspace of overconvergent forms, the critical point must appear in the eigenvariety of modular forms. Such critical points were studied in \cite{BH17}. 
\end{proof}

Now if $\cV^+$ is a component of $\cE_{\Omega^+,h^+}^{\bullet,+}$ through $x_{f^+}$, then its transfer $\cV \defeq \mathrm{BC}(\cV^+) \subset \cE_{\cU,h}^\bullet$ is a family through $x_f$ over $\Omega = \mathrm{Im}(\Omega^+) \subset \cW$. In particular, writing $\Omega = \Sp(\cO(\cU)/\fP)$, we see that $\fP \subset \m_{f}$.

What this does \emph{not} tell us, however, is whether $\cV$ is itself an irreducible component of $\cE_{\cU,h}^{\bullet}$, or if it is in fact a proper subspace of a (higher-dimensional) component. Algebraically, this would mean that $\fP$ is not a minimal prime in $\m_{f}$. The non-abelian Leopoldt conjecture (Conjecture \ref{conj:dim}) predicts that in fact, $\cV$ will always be a component.

\subsection{Top row II: descent to bottom degree}\label{sec:descent}
The construction of $p$-adic $L$-functions in the present setting is performed in the bottom degree $d$ in which automorphic representations appear. The appearance of $\m_f$ in $\hc{d}(Y_1(\n),\sD_{\cU})$ for some affinoid $\cU$, however, is far from obvious, since the general formalism is not precise about degrees. To ensure the existence of a class, for the rest of the paper, \emph{we assume that Conjecture \ref{conj:dim} holds}. In particular, since in our setting the top and bottom degrees are $2d$ and $d$ respectively, and $\mathrm{dim}(\cW^{\mathrm{full}}) = 2d$, every component through $f$ has dimension $d$.

Let now $\cU \subset \cW^{\mathrm{full}}$ be an affinoid neighbourhood of $\lambda$ with maximal dimension (that is, $\mathrm{dim}(\cU) = \mathrm{dim}(\cW^{\mathrm{full}})$). Let $\m_{f} \subset \bT_{\cU,h}^\bullet$ be the maximal ideal corresponding to $f$, which exists by Theorem \ref{thm:total cohom}, and let $\fp_{\mathrm{min}} \subset \m_{f}$ be a minimal prime ideal, corresponding to an irreducible component $\cV$ of the eigenvariety through $x_f$. Let $\fP_{\mathrm{min}} = \cO(\cU)\cap \fp_{\mathrm{min}}$ denote its contraction to $\cO(\cU)$, so that $\cV$ lies over the subspace $\Omega \defeq \Sp(\cO(\cU)/\fP_{\mathrm{min}}) \subset \cW$. 

The localisation $\cO(\cU)_{\fP_{\mathrm{min}}}$ is a regular local ring, and its maximal ideal $\fP_{\mathrm{min}}$ is generated by a regular sequence $(x_1,x_2,...,x_{d})$, where the sequence has $d$ terms since we assume Conjecture \ref{conj:dim} holds for $\cV$. Up to shrinking $\cU$, we may assume that each of the $x_i$ is in $\cO(\cU)$. Define $\cU^0 \defeq \cU$, and for $i = 1,...,d$, let
\[
\cU^i \defeq \Sp\big[\cO(\cU)/(x_1,...,x_i)\big].
\]
Then $\cU = \cU^0 \supset \cU^1 \supset \cdots \supset \cU^{d} = \Omega$.

The following non-vanishing, due to Newton \cite[App.~B]{Han17}, is the reason for assuming the non-abelian Leopoldt conjecture.

\begin{proposition}\label{prop:non-vanishing}
	Suppose Conjecture \ref{conj:dim} holds for $\cV$. Then for each $i = 0,...,d$, we have
	\[
	\hc{2d-i}(Y_1(\n),\sD_{\cU^i})_{f}^{\leq h} \neq 0.
	\]
\end{proposition}
\begin{proof}
	This is precisely \cite[Prop.~B.3]{Han17}, combined with the remark at the end of App.~B \emph{op.\ cit}., which means that if $\mathrm{dim}(\cV) = d$, the minimal degree $r$ such that $\hc{r}(Y_1(\n),\sD_{\cU})_{f}^{\leq h} \neq 0$ is equal to the top degree $2d$. 
\end{proof}
\begin{remark}
	We briefly indicate \emph{why} Conjecture \ref{conj:dim} is required. Generally, when working in the top degree, specialisation of the weight is surjective on the cohomology, since the codomain is in $\hc{2d+1} = 0$. This gives the non-vanishing of $\hc{2d}(Y_1(\n),\sD_{\cU})^{\leq h}_{f}$ by Nakayama. Then for each $i$, the short exact sequence $0 \to \cD_{\cU^i} \to \cD_{\cU^i} \to \cD_{\cU^{i+1}} \to 0$ gives rise to a connecting surjective map
	\[
	\hc{j}(Y_1(\n),\sD_{\cU^{i+1}})^{\leq h}_{f} \longrightarrow \hc{j+1}(Y_1(\n),\sD_{\cU^i})^{\leq h}_{f}[x_{i+1}].
	\]
	For appropriate $j$, as in the theorem, one shows this is an isomorphism: so `torsion in degree $j+1$ gives classes in degree $j$'. The conjecture then implies that `the torsion is maximal', so that we can descend maximally to bottom degree.
\end{remark}

\begin{definition}
	Let $\fP \subset \cO(\cU)$ be a prime, and define $\Sigma \defeq \Sp(\cO(\cU)/\fP)$. We say $f$ \emph{varies in a family over $\Sigma$} if there exists a prime $\pri \subset \m_{f}$ such that $\fP = \pri \cap \cO(\Sigma)$. The family is maximal if $\pri$ is a minimal prime.
\end{definition}

Let $\Omega \subset \cW$ be small enough that $f$ varies in a family over $\Omega$; then by definition, if $\Sigma\subset\Omega$ is any Zariski-closed subspace containing $\lambda$, then $f$ varies in a family over $\Sigma$. (In the sequel, we wish to apply this to one-dimensional affinoids inside $\Omega$). If $\lambda \in \cW$, we use a subscript  $\cO(-)_\lambda$ for the (algebraic) localisation of $\cO(-)$ at $\m_\lambda$.

\begin{corollary}\label{cor:non-vanishing}
	Suppose Conjecture \ref{conj:dim} holds, and $f$ varies in a family over an affinoid $\Sigma\subset\cW^{\mathrm{full}}.$ Then
	\[
	\hc{d}(Y_1(\n),\sD_{\Sigma})^{\leq h}_{f} \neq 0.
	\]
\end{corollary}

\begin{proof}
	From the definition, $\Sigma = \mathrm{Sp}(\cO(\cU)/\fP)$ for some prime $\fP$. Let $\fp_{\mathrm{min}}$ be a minimal prime of $\bT_{\cU,h}^\bullet$ contained in $\fP \cdot \bT_{\cU,h}^\bullet$, which is thus contained in $\m_{f}$ by definition of $\fP$. Let $\fP_{\mathrm{min}}$ denote its contraction to $\cO(\cU)$, and let $\Omega \defeq \Sp(\cO(\cU)/\fP_{\mathrm{min}})$. By Conjecture \ref{conj:dim} and Proposition \ref{prop:non-vanishing}, we know $\hc{d}(Y_1(\n),\sD_\Omega)^{\leq h}_{f} \neq 0$.
	
	The localisation $\cO(\cU)_{\fP}$ is a regular local ring with maximal ideal $\fP\cO(\cU)_{\fP}$; let $(x_1,...,x_r)$ be a regular sequence generating this ideal. Up to shrinking $\cU$, we may suppose $x_i \in \cO(\cU)$. Necessarily $r \geq d$, and we may choose this sequence so that $(x_1,...,x_d)$ generates $\fP_{\mathrm{min}}$ (and is thus the regular sequence chosen in the proof of \cite[Prop.~B.3]{Han17}). For $i = 0,...,r-d$, let
	\[
	\Omega^i \defeq \Sp\big[\cO(\cU)/(x_1,...,x_{d+i})\big].
	\]
	Then $\Omega^0 = \Omega$ and $\Omega^{r-d} = \Sigma$. For each $i \leq r-d-1$, we have a short exact sequence $0 \to \cD_{\Omega^i} \xrightarrow{x_{i+1}} \cD_{\Omega^i} \to \cD_{\Omega^{i+1}} \to 0$, giving a long exact sequence
	\[
	\cdots \to \hc{d}(Y_1(\n),\sD_{\Omega^i}) \xrightarrow{x_{i+1}} \hc{d}(Y_1(\n),\sD_{\Omega^i}) \longrightarrow \hc{d}(Y_1(\n),\sD_{\Omega^{i+1}}) \to \cdots,
	\]
	which, after taking small slopes, localising and truncating, gives an injection
	\begin{equation}\label{eqn:injection truncated}
		0 \to \hc{d}(Y_1(\n),\sD_{\Omega^i})^{\leq h}_{f} \otimes_{\cO(\Omega^i)_\lambda}\cO(\Omega^i)_\lambda/(x_{i+1}) \hookrightarrow \hc{d}(Y_1(\n),\sD_{\Omega^{i+1}})^{\leq h}_{f}.
	\end{equation}
	Now suppose $\hc{d}(Y_1(\n),\sD_{\Sigma})^{\leq h}_{f}$ is zero. Since $\Sigma = \Omega^{r-d}$, by \eqref{eqn:injection truncated} and Nakayama's lemma we see that $\hc{d}(Y_1(\n),\sD_{\Omega^{r-d-1}})^{\leq h}_{f}  = 0.$ Continuing by induction, since $\Omega^0 = \Omega$ we conclude that 
	\[
	\hc{d}(Y_1(\n),\sD_{\Omega})^{\leq h}_{f} = 0,
	\]
	which contradicts Proposition \ref{prop:non-vanishing}. Thus we have the required non-vanishing.
\end{proof}

Note that by Theorem \ref{t:transfert}, Proposition \ref{prop:pt} and the veracity of Conjecture \ref{conj:dim} for $G^+$ (since in this case $\ell = 0$), we know that $f$ varies in maximal-dimension families over $\cW$.

\subsection{\'Etaleness over smooth curves}\label{sec:etale}
Let $\Sigma \subset \cW$ be a one-dimensional neighbourhood of $\lambda$, and assume it is smooth at $\lambda$. (For example, $\Sigma$ could be a neighbourhood in any line through $\lambda$). 

\begin{proposition}\label{prop:free rank one}
	The space $\hc{d}(Y_1(\n),\sD_\Sigma)^{\leq h}_{f}$ is free of rank one over $\cO(\Sigma)_\lambda$. 
\end{proposition}

\begin{proof}
	Since $\Sigma$ is a curve smooth at $\lambda$, the maximal ideal $\m_\lambda \subset \cO(\Sigma)_\lambda$ is principal. Let $m$ be a generator. Then multiplication by $m$ induces a truncated long exact sequence
	\[
	0 \to \hc{d}(Y_1(\n),\sD_{\Sigma})^{\leq h}_{f} \otimes_{\cO(\Sigma)_\lambda}\cO(\Sigma)_\lambda/\m_\lambda \hookrightarrow \hc{d}(Y_1(\n),\sD_{\lambda})^{\leq h}_{f}.
	\]
	exactly as in \eqref{eqn:injection truncated}. The right-hand side is isomorphic to $\hc{d}(Y_1(\n),\sV_\lambda^\vee)_{f}^{\leq h}$, since $f$ is non-critical; but our assumptions on $f$ ensure that this is a one-dimensional $L$-vector space \eqref{eqn:classical line}. It follows that the left-hand side is either 0 or the map is an isomorphism. But if it were zero, Nakayama's lemma would imply that $\hc{d}(Y_1(\n), \sD_{\Sigma})^{\leq h}_{f} = 0$, which contradicts Corollary \ref{cor:non-vanishing}. So the map is an isomorphism and 
	\[
	\mathrm{dim}_L \big[\hc{d}(Y_1(\n),\sD_{\Sigma})^{\leq h}_{f} \otimes_{\cO(\Sigma)_\lambda}\cO(\Sigma)_\lambda/\m_\lambda\big] = 1.
	\]
	Nakayama's lemma then implies that $\hc{d}(Y_1(\n),\sD_\Sigma)^{\leq h}_{f}$ is cyclic over $\cO(\Sigma)_\lambda$, of the form $\cO(\Sigma)_\lambda/I$ for some ideal $I$ inside the maximal ideal $\m_\lambda$. 
	
	To conclude, it suffices to prove that this module is torsion-free, since then $I=0$. We do this in the next lemma.
\end{proof}

\begin{lemma}\label{lem:torsion-free}
	Let $\Omega$ be any affinoid containing $\lambda$. The space $\hc{d}(Y_1(\n),\sD_{\Omega})^{\leq h}_{f}$ is torsion-free.
\end{lemma}
\begin{proof}
	The result is trivial if the module is 0, so assume it is non-zero; it then suffices to prove it is $\m_\lambda$-torsion-free. Let $y \in \m_\lambda$ be any element, and define $\Omega_y \defeq \Sp(\cO(\Omega)/(y))$. By \cite[Prop.~4.5.2]{Han17}, we know that for any $j < d$ we have
	\begin{equation}\label{eqn:zero below}
		\hc{j}(Y_1(\n),\sD_{\Omega_y})^{\leq h}_{f} = 0.
	\end{equation}

	Now consider the short exact sequence $0 \to \cD_{\Omega} \xrightarrow{y} \cD_\Omega \to \cD_{\Omega_y} \to 0$. The localised long exact sequence attached to this truncates to 
	\[
	\hc{d-1}(Y_1(\n),\sD_{\Omega_y})^{\leq h}_{f} \to \hc{d}(Y_1(\n),\sD_{\Omega})^{\leq h}_{f} \xrightarrow{y} \hc{d}(Y_1(\n),\sD_{\Omega})^{\leq h}_{f}.
	\]
	The first term is zero by \eqref{eqn:zero below}; so multiplication by $y$ is injective on the cohomology, which is therefore $y$-torsion-free. We conclude since $y$ was arbitrary.
\end{proof}

In the case required in the proof of Proposition \ref{prop:free rank one}, that is for $\Omega = \Sigma$, we need only check the torsion-free property for $y = m$, the generator of $\m_\lambda$ from above; and then $\Omega_y = \{\lambda\}$, and we know $\hc{j}(Y_1(\n),\sD_\lambda)_{f}^{\leq h} = 0$ simply by the non-criticality of $f$, without appealing to \cite{Han17}.

Drawing the results of the previous two sections together, we have shown:

\begin{corollary}\label{cor:etale}
	Assume Conjecture \ref{conj:dim}. Let $\Sigma \subset \cW$ be a curve smooth at $\lambda$. \begin{enumerate}[(1)]\setlength{\itemsep}{0pt}
		\item $\hc{d}(Y_1(\n),\sD_\Sigma)^{\leq h}_f$ is free of rank one over $(\bT_{\Sigma,h}^d)_f$.
		\item There is a point $x_{f}\in \cE_{\Sigma,h}^d$, at which the weight map $\cE_{\Sigma,h}^d \to \Sigma$ is \'etale. 
		\item There is a unique component $\cV$ of $\cE_{\Sigma,h}^d$ through $x_{f}$, and it has dimension one.
	\end{enumerate}
\end{corollary}

\begin{proof}
	Since $\hc{d}(Y_1(\n),\sD_\Sigma)^{\leq h}_{f}$ is free of rank one over $\cO(\Sigma)_\lambda$, its endomorphism ring is $\cO(\Sigma)_\lambda$, and we deduce that $(\bT_{\Sigma,h}^d)_{f} \cong \cO(\Sigma)_\lambda$. This gives (1) and the existence of the point $x_{f}$.
	
	We turn to \'etaleness of the weight map. Until now all of our localisations have been algebraic. To derive rigid geometric consequences we would like instead to study rigid localisations.  The results of \cite[\S7.3.2(3)]{BGR} say that if $X = \mathrm{Sp}(A)$ is a rigid space, then the rigid localisation of $A$ at a point $x$ is always a faithfully flat (Noetherian) extension of the algebraic localisation of $A$ at $\m_x$, and that these local rings have the same residue field and completions. 
	
	By (1), the map $\cO(\Sigma) \to \cO(\cE_{\Sigma,h}^d) = \bT_{\Sigma,h}^d$ is \'etale after algebraically localising at $\lambda$ and $x_f$ respectively. This map is of finite type by the slope condition. Since a finite type local homomorphism of Noetherian local rings is \'etale if and only if the induced map on completions is \'etale, and the rigid and algebraic local rings have the same completions, we thus obtain rigid \'etaleness from algebraic \'etaleness. Then (3) is a direct consequence of (2).
\end{proof}

\begin{remark}
	We remark that we are in a somewhat surreal situation in which having \emph{less} variation is important in controlling $p$-adic families.
	Knowledge of Conjecture \ref{conj:dim} also has very nice consequences in the study of a conjecture of Venkatesh on the freeness of the total cohomology over derived Hecke algebra: in \cite[Cor.\ 4.10]{HT17}, Hansen and Thorne prove a $p$-adic 
	realisation of his motivic conjecture under its assumption. In particular, they prove that the total cohomology at $\refpi$, with $p$-adic coefficients, is free of rank one over $\wedge^\bullet V_\pi$, where $V_\pi \defeq \ker(\cO(\cU)_\lambda \to \bT_{\cU,\refpi}) \otimes \cO(\cU)/\m_\lambda$; note that under the conjecture, $V_\pi$ has dimension $\ell$. Under some additional assumptions, they show that $V_\pi$ is a $p$-adic Selmer group.
\end{remark}

\begin{remark}\label{rem:2-dim}
	If $E$ is imaginary quadratic, then Conjecture \ref{conj:dim} is true, the families are all one-dimensional, and this is the whole story. If $E$ is larger, then we always have more variation, and this \'etaleness result can be lifted to at least two-dimensional neighbourhoods. Indeed, let $\Omega \supset \Sigma \supset \{\lambda\}$ inside $\cW$ be smooth at $\lambda$. By repeating the arguments above, we can show 
	\begin{equation}\label{eqn:injection higher}
		\hc{d}(Y_1(\n),\sD_\Omega)^{\leq h}_{f} \otimes_{\cO(\Omega)_\lambda}\cO(\Omega)_\lambda/(n) \hookrightarrow \hc{d}(Y_1(\n),\sD_\Sigma)^{\leq h}_{f},
	\end{equation}
	where $\m_\lambda = (m,n) \subset \cO(\Omega)_\lambda$. Moreover the left-hand side is non-zero by Corollary \ref{cor:non-vanishing}. Since the right-hand side is isomorphic to $\cO(\Sigma)_\lambda$ -- a principal ideal domain -- the left-hand side is cyclic, and Nakayama says that $\hc{d}(Y_1(\n),\sD_{\Omega})^{\leq h}_{f}$ is also cyclic over $\cO(\Omega)_\lambda$. But it is also torsion-free by Lemma \ref{lem:torsion-free}, hence free of rank one.
	
	Going beyond this appears to be difficult, however, since it is hard to control the cokernel of the injection \eqref{eqn:injection truncated} in general (See also Remark \ref{rem:higher dims}). For the $p$-adic Artin formalism of \S\ref{sec:artin formalism}, $1$-dimensional variation is enough.
\end{remark}

\subsection{Families of $p$-adic $L$-functions}\label{sec:CM families p-adic}
Let $\Sigma$ be a smooth curve through $\lambda$ as above, let $\cV$ be the unique component through $x_{f}$ in $\cE_{\Sigma,h}^d$. We know from Corollary \ref{cor:etale} that 
\begin{equation}\label{eq:free rank one}
	\hc{d}(Y_1(\n),\sD_\Sigma)_f^{\leq h}\defeq \hc{d}(Y_1(\n),\sD_\Sigma)^{\leq h}\otimes_{\bT_{\Sigma,h}^d} (\bT_{\Sigma,h}^d)_f \text{ is free of rank one over }(\bT_{\Sigma,h}^d)_f,
\end{equation}
where both sides are algebraic localisations. We would like to lift this to an affinoid neighbourhood. Thus consider the rigid analytic localisation
\[
\bT_{x_f} \defeq \varinjlim\limits_{x_f \in \Sp(T) \subset \cV} T,
\]
which by \cite[\S7.3.2]{BGR} is faithfully flat over $(\bT_{\Sigma,h}^d)_f$. Overconvergent cohomology defines a rigid coherent sheaf $\cF$ on $\cE_{\Sigma,h}^d$ \cite[Thm.\ 4.2.2]{Han17}; tensoring \eqref{eq:free rank one} over $(\bT_{\Sigma,h}^d)_f$ with $\bT_{x_f}$, the rigid stalk
\[
\cF_{x_f} = \hc{d}(Y_1(\n),\sD_{\Sigma})^{\leq h} \otimes_{\bT_{\Sigma,h}^d} \bT_{x_f}\text{ is free of rank one over }\bT_{x_f}.
\]
Then possibly after shrinking $\cU$ and $\cV$, we can find a direct summand $T$ of $\bT_{\Sigma,h}^{d}$ such that:
\begin{itemize}\setlength{\itemsep}{0pt}
	\item $\hc{d}(Y_1(\n),\sD_{\Sigma})^{\leq h}\otimes_{\bT_{\Sigma,h}^{d}} T$ is free of rank one over $T$ (since
	by \cite[Lem.\ 2.10]{BDJ17}, an isomorphism on rigid stalks lifts to an isomorphism on neighbourhoods),
	\item $\cV = \Sp(T)$ is affinoid, and \'etale over $\Sigma$ (since \'etaleness is an open condition).
\end{itemize}
Indeed, $T$ is always a quotient of $\bT_{\Sigma,h}^d$, and \'etaleness ensures it is a summand. It follows that 
\[
\hc{d}(Y_1(\n),\sD_{\Sigma})^{\leq h}\otimes_{\bT_{\Sigma,h}^{d}} T \subset \hc{d}(Y_K,\sD_{\Sigma})^{\leq h}
\]
is also a summand.	

Let now $\Phi_{\cV}$ be a generator of this module. Note that after possibly further shrinking $\cV$, we may assume that every classical $y \in \cV$ is non-critical and $\mathrm{dim}_L \hc{d}(Y_1(\n),\sV_{\lambda_y}^\vee)_{f_y} = 1,$ where $y$ corresponds to a refinement $f_y$ of an automorphic representation of weight $\lambda_y$.

\begin{proposition}
	For any classical $y \in \cV(L)$, there exists a $p$-adic period $c_y \in L$ such that 
	\[
	\mathrm{sp}_{\lambda_y}(\Phi_{\cV}) = c_y\cdot \Phi_{f_y}.
	\]
\end{proposition}
\begin{proof}
	Let $y$ be such a point, corresponding to a maximal ideal $\m_y$. Then $\hc{d}(Y_1(\n),\sD_{\lambda_y}(L))^{\leq h}_{f_y}$ is a one-dimensional $L$-vector space, as the classical cohomology is a line and $f_y$ is non-critical. It is spanned by $\Phi_{f_y}$, a lift of a classical generator. The methods above show that
	\begin{align*}
		\hc{d}(Y_1(\n),\sD_{\Sigma})^{\leq h}\otimes_{\bT_{\Sigma,h}^{d}} T/\m_y &\cong \hc{d}(Y_1(\n),\sD_{\Sigma})_{f_y}^{\leq h}\otimes_{\cO(\Sigma)_\lambda}\cO(\Sigma)_\lambda/\m_{\lambda_y}\\
		&\cong \hc{d}(Y_1(\n),\sD_\lambda)^{\leq h}_{f_y},
	\end{align*}
	where the last isomorphism was proved in Proposition \ref{prop:free rank one}. Thus the image of $\Phi_{\cV}$ under $\mathrm{sp}_{\lambda_y}$ also lands in this $L$-line. But then the result follows.
\end{proof}

\begin{remark}
	Note we can always shrink yet further so that $c_y \neq 0$ for all $y$. The class $\Phi_{f_y}$ itself is only well-defined up to scaling by an element of a number field $\Q(f_y)$ over which the Hecke eigenvalues are defined, the indeterminacy corresponding to scaling the periods. It is natural to ask if the $c_y$ themselves can be $p$-adically interpolated. An approach to this might be to consider overconvergent Eichler--Shimura isomorphisms (as in \cite{AIS15}).
\end{remark}

In \S\ref{sec:top row II}, we described evaluation maps $\mathrm{Ev}^{\fc}_{\lambda}$ and 
\begin{align*}
	\mathrm{Ev}^{\fc}_{\cV} : \hc{d}(Y_1(\n),\sD_{\Sigma})^{\leq h} \otimes_{\cO(\Sigma)}\cO(\cV) &\to \cD(\Gal_p,\cO(\Sigma))\otimes_{\cO(\Sigma)} \cO(\cV)\\
	&\cong \cD(\Gal_p,\cO(\cV)),
\end{align*}
related by
\[
\mathrm{sp}_y \circ \mathrm{Ev}^{\fc}_{\cV} = \mathrm{Ev}^{\fc}_{\lambda_y} \circ \mathrm{sp}_{\lambda_y}.
\]
(Here, since we have $\cV \to \Sigma$ \'etale, we do not need to be careful about whether we take $\mathrm{sp}_\lambda$ or $\mathrm{sp}_y$, but in non-\'etale situations more care is needed). 

\begin{definition}
	Define the $p$-adic $L$-function over $\cV$ to be $L_p(\cV) \defeq \alpha_p^{-1}\mathrm{Ev}^{(p)}_{\cV}(\Phi_{\cV})$ where $\alpha_{p}$ is the $U_p$ eigenvalue of $\Phi_{\cV}$.
\end{definition}
Then 
\begin{align*}
	\mathrm{sp}_y(L_p(\cV)) &= \alpha_p^{-1}(y)\mathrm{sp}_y\circ\mathrm{Ev}^{(p)}_{\cV}(\Phi_{\cV})\\
	&= \alpha_p^{-1}(y)\mathrm{Ev}^{(p)}_{\lambda_y} \circ \mathrm{sp}_{\lambda_y}(\Phi_{\cV}) = \alpha_p^{-1}(y)c_y\cdot \mathrm{Ev}^{(p)}_{\lambda_y}(\Phi_{f}) = c_y\cdot L_p(f),
\end{align*}
so that $L_p(\cV)$ interpolates the $p$-adic $L$-functions of classical points in $\cV$, as required. To recover the formulation of Theorem \ref{thm:intro} in the introduction, we identify $\cO(\cV^+)$ and $\cO(\cV)$ under the base-change map, and use the Amice transform to identify $\cD(\Gal_p,\cO(\cV^+))$ with $\cO(\cV^+\times \sX(\Gal_p))$.

This $p$-adic $L$-function has $d+2+\delta_E$ variables, where $\delta_E$ is the Leopoldt defect, which Leopoldt's conjecture predicts is zero. These comprise one weight variable and $d+\delta_E+1$ variables over the Galois group $\mathrm{Gal}_p$.

\begin{remark}\label{rem:higher dims}
	We comment on variation in higher-dimensional families. In Remark \ref{rem:2-dim}, we described an \'etaleness result for smooth 2-dimensional affinoids $\Omega$ in weight space, which gives a canonical class $\Phi_{\Omega}$ in $\hc{d}(Y_1(\n),\sD_\Omega)_f^{\leq h}$. However, this does not immediately translate into variation of $p$-adic $L$-functions over such $\Omega$. One needs to prove that the natural map
	\begin{equation}\label{eq:2-dim spec}
		\hc{d}(Y_1(\n),\sD_{\Omega})_f^{\leq h}  \longrightarrow \hc{d}(Y_1(\n),\sD_\lambda)_f^{\leq h}
	\end{equation}
	is surjective, so that $\Phi_{\Omega}$ does interpolate the classes $\Phi_{f_y}$ over $\Omega$. For this one needs to rule out the possibility that 
	\[
	\hc{d}(Y_1(\n),\sD_\Omega)^{\leq h}_f\otimes_{\cO(\Omega)_\lambda} \cO(\Sigma)_\lambda \hookrightarrow \hc{d}(Y_1(\n),\sD_\Sigma)^{\leq h}_f \cong \cO(\Sigma)_\lambda
	\]
	lands in a proper submodule of $\cO(\Sigma)_\lambda$, which would then be inside the maximal ideal $\m_\lambda$ and be mapped to 0 under further specialisation. In this case, the image of \eqref{eq:2-dim spec} would be 0.
	
	More generally, for an arbitrary affinoid $\cU$, if we could prove
	\[
	\hc{d}(Y_1(\n),\sD_{\cU})^{\leq h}_f \otimes_{\cO(\cU)_\lambda}\cO(\cU)_\lambda/\m_\lambda \cong \hc{d}(Y_1(\n),\sD_\lambda)^{\leq h}_f,
	\]
	then the methods of this section would be enough to show both \'etaleness over $\cU$, giving a class $\Phi_{\cU,f}$, \emph{and} specialisation of $\Phi_{\cU,f}$ at classical points, giving a $p$-adic $L$-function varying over all of $\cU$. When $f$ appears in only one degree of cohomology, the Tor spectral sequence $\mathrm{Tor}_{-i}^{\cO(\cU)}(\hc{j}(Y_1(\n),\sD_{\cU})_f^{\leq h}, \cO(\cU)_\lambda/\m_\lambda) \implies \hc{i+j}(Y_1(\n),\sD_\lambda)_f^{\leq h}$ degenerates at the $E_2$ page and gives the required isomorphism when combined with Corollary \ref{cor:non-vanishing}. This is not true for the $E_2^{0,d}$ term in general, due to the non-vanishing of higher cohomology and Tor groups.
\end{remark}

\subsection{The non-base-change case}
We remark briefly on the more general case where $f$ is non-critical but not necessarily base-change. In this case, much of the above goes through if we work with $\cW^{\mathrm{full}}$ instead of the base-change (conjugate-invariant) weights $\cW$.
The general results of \S\ref{sec:local geometry} still yield a point $x_f \in \cE_{\cU,h}^\bullet$; assuming non-abelian Leopoldt we may still exhibit a family through $x_f$ lying over a $d$-dimensional subspace $\Omega \subset \cW^{\mathrm{full}}$; and the results of \S\ref{sec:descent} yield non-vanishing of $\hc{d}(Y_K,\sD_{\Omega})_{f}$. If $\Sigma \subset \Omega$ is a curve containing (and smooth at) $\lambda$, then the results of \S\ref{sec:etale} then show that there is a unique component $\cV$ of $\cE_{\Sigma,h}^d$ through $x_f$, which moreover is \'etale at $x_f$. Then \S\ref{sec:CM families p-adic} provides the construction of a $p$-adic $L$-function $L_p(\cV)$ over $\cV$, which specialises to $c_y \cdot L_p(f_y)$ at any classical point $y \in \cV$ (and in particular at $x_f$).

The base-change hypothesis above is used \emph{only} to control $\Omega$. In this case the $p$-adic base-change transfer of \S\ref{sec:base change transer} shows that $\Omega$ can be taken to be an affinoid in $\cW$, which is smooth at $\lambda$, and allows us to parametrise explicit smooth curves $\Sigma$ through $x_f$. Moreover we may always take $\Sigma$ to contain a Zariski-dense set of classical points. If we drop the base-change hypothesis, then it seems very difficult to prove any smoothness result, and further we cannot guarantee that $\Omega$ contains \emph{any} pure weight except for $\lambda$. In particular (even assuming smoothness) the values of $L_p(\cV)$ will not in general compute classical $L$-values. Such non-classical families -- and their $p$-adic $L$-functions -- remain very mysterious, and it would be particularly interesting to have a better understanding of their arithmetic (and what they can tell us about $f$).

\section{$p$-adic Artin formalism}\label{sec:artin formalism}
\subsection{Statement}
Let $f^+$ be a Hilbert modular form over $E^+$, and $f$ its base-change to $E$. We maintain the assumptions as in \S\ref{sec:preliminaries}, and assume that $f^+, f^+\otimes\chi_{E/E^+}$ and $f$ are non-critical, where $\chi_{E/E^+}$ is the quadratic Hecke character of $E^+$ attached to $E/E^+$. By Artin formalism we have
\[
L(f,\psi,s) = L(f^+,\psi,s)L(f^+\otimes \chi_{E/E^+},\psi,s).
\]
In Tate's formulation, this is $L(f,\varphi\circ N_{E/E^+}) = L(f^+\otimes \chi_{E/E^+},\varphi)L(f^+,\varphi)$, where $\varphi$ is a Hecke character of $E^+$. In this section, we prove a $p$-adic analogue. 

The overconvergent cohomology construction of $p$-adic $L$-functions of Hilbert forms was given in \cite{Bar15}, and varied in families in \cite{BDJ17,BH17}. As $\chi_{E/E^+}$ is quadratic, the slopes at $p$ of $f^+$ and $f^+\otimes\chi_{E/E^+}$  are equal, so $f^+\otimes \chi_{E/E^+}$ is non-critical slope if and only if $f^+$ is.

Let $E^{+,\cyc} = E^+(\mu_{p^\infty})$; then there is a natural projection $\Gal_p^+ \rightarrow\Gal_p^{+,\cyc} \defeq \Gal(E^{+,\cyc}/E^+)$. For the rest of the paper, we abuse notation and write $L_p(f^+)$ (resp.\ $L_p(f^+\otimes\chi_{E/E^+})$) for the distribution over $\Gal_p^{+,\cyc}$ obtained from the $p$-adic $L$-function of $f^+$ (resp.\ $f\otimes\chi_{E/E^+}$). 
\begin{definition}
	Let $\mathrm{pr}_{E/E^+} : \Gal_p \to \Gal_p^{+,\cyc}$ be the natural projection. Define $L_p^{\cyc}(f) \in \cD(\Gal_p^{+,\cyc}, L)$, the \emph{restriction to the cyclotomic line}, by
	\[
	\int_{\Gal_p^{+,\cyc}} \varphi  \cdot dL_p^\cyc(f) \defeq \int_{\Gal_p} (\varphi \circ \mathrm{pr}_{E/E^+}) \cdot dL_p(f),
	\]
	for $\varphi \in \cA(\Gal_p^{+, \cyc}, L)$.
\end{definition}

\begin{theorem}\label{thm:artin formalism}
	Suppose the non-abelian Leopoldt conjecture holds for $f$. Then
	\[
	L_p^{\cyc}(f) = L_p(f^+)\cdot L_p(f^+\otimes\chi_{E/E^+})
	\]
	as distributions on $\Gal_p^{+,\cyc}$.
\end{theorem}

Here the non-abelian Leopoldt conjecture is required for the existence of the base-change family of $p$-adic $L$-functions above.

\begin{remark}
	We restrict to $\Gal_p^{+,\cyc}$ to ensure it is a one-dimensional $p$-adic group, so that distributions on it can be uniquely determined by admissibility and interpolation properties \cite{AV75}. Note that Leopoldt's conjecture would imply that $\mathrm{Gal}_p^+$ is, in any case, only a finite extension of $\mathrm{Gal}_p^{+,\cyc}$, and hence is itself one-dimensional. In particular, given the (usual) Leopoldt conjecture, one could replace every instance of $\Gal_p^{+,\cyc}$ in this section with $\Gal_p^+$.
\end{remark}

\subsection{The theorem for extremely small slope}
The following is a more-or-less automatic, but key, first step in proving Theorem \ref{thm:artin formalism}.
\begin{proposition}\label{prop:equal interp}
	The distributions $L_p^\cyc(f)$ and $L_p(f^+)L_p(f^+\otimes\chi_{E/E^+})$ agree at any critical Hecke character.
\end{proposition}
\begin{proof}
	By classical Artin formalism, the complex $L$-values agree. One can choose the complex periods in a compatible way. Thus one needs to check that the $p$-adic interpolation factors of the two sides agree, which is an exercise in book-keeping. The most difficult step is the compatibility of Gauss sums, which is a characteristic $0$ version of the classical Hasse--Davenport identity; see \cite[\S6, Cor.\ 1]{Mar72}. In the Bianchi case, this is explained in detail in \cite[\S7]{BW18}.
\end{proof}

Define a renormalised $U_p$-operator $U_p^\circ \defeq p^{-v}U_p$ to ensure its integrality, as in the proof of \cite[Prop.~9.4]{BW_CJM}. Let $h_p^\circ$ be the slope of the $U_p^\circ$-eigenvalue of $f^+$, which is constant in families. We say $f^+$ has \emph{extremely small slope} if $h_p^\circ < \tfrac{1}{2}\mathrm{min}(k_\sigma + 1)$.

\begin{proposition}\label{prop:extremely}
	If $f^+$ has extremely small slope, then Theorem \ref{thm:artin formalism} holds for $f^+,f$.
\end{proposition}
\begin{proof} The distributions $L_p(f^+)$ and $L_p(f^+\otimes \chi_{E/E^+})$ are admissible of order $h_p^\circ$, then $L_p(f^+)L_p(f^+\otimes \chi_{E/E^+})$ is of order $2h_p^{\circ} < \mathrm{min}(k_\sigma + 1)$. Moreover, as the slope of $f$ is twice the slope of $f^+$ (as $\alpha_p(f) = \alpha_p(f^+)^2$) then $L_p^\cyc(f)$ is also admissible of order $2h_p^\circ$.  Then by \cite{AV75} they are uniquely determined by their interpolation properties, and we conclude using \ref{prop:equal interp}.
\end{proof}

When $f^+$ does not have extremely small slope, the argument no longer works, as both sides are not uniquely determined by the interpolation. We instead argue using $p$-adic families.

\subsection{Twisted $p$-adic $L$-functions}
\label{sec:twisted}

Let $\cV^+$ be a family through $f^+$ in an eigenvariety $\cE_{\Sigma^+,h^+}^d$. Note that $d$ is the only degree of cohomology in which $f^+$ appears, so this makes sense. From \cite{BDJ17,BH17} we obtain a class
\[
\Phi_{\cV^+} \in \hc{d}(Y_1(\n^+),\sD_{\Sigma^+})^{\leq h^+}_{f^+},
\]
and the evaluation maps into $\cD(\Gal_p^+,\cO(\cV^+))$, as required to construct $L_p(\cV^+)$ via the methods of \S\ref{sec:survey}. Since $\cV^+$ is \'etale over $\Sigma^+$ \cite{BDJ17}, we consider this to be valued in $\cO(\Sigma^+)$. As $f^+\otimes\chi_{E/E^+}$ is non-critical, up to shrinking $\Sigma^+$ there exists a family $\sV^+$ through it, \'etale over $\Sigma^+$. 

\begin{lemma}
	The family $\sV^+$ is $\cV^+\otimes\chi_{E/E^+}$ i.e. each classical point of $\sV^+$, corresponding to a modular form $g'$ of weight $\mu$, is of the form $g \otimes\chi_{E/E^+}$ for $g$ the point of $\cV$ above $\mu$. 
\end{lemma}

\begin{proof}
	We again use $p$-adic Langlands functoriality \cite[Thm.~3.2.1]{JoNew}. Let $\fd = \mathrm{cond}(\chi_{E/E^+}) = \mathrm{disc}(E/E^+) \subset \cO_{E^+}$. Classically, twisting by $\chi_{E/E^+}$ gives a map
	\[
	-\otimes \chi_{E/E^+} : S_{\lambda^+}(K_1(\n^+)) \longrightarrow S_{\lambda^+}(K_1(\n^+\fd^2)),
	\]
	recalling that $\n^+$ and $\fd$ are assumed coprime. On Hecke eigenvalues, the $T_v$-eigenvalues are multiplied by $\chi_{E/E^+}(v)$, the $U_{\pri}$-eigenvalues by $\chi_{E/E^+}(\pri)$ and the eigenvalues of the diamond operator $\langle u\rangle$ are multiplied by $\chi_{E/E^+}(u)^2$ (i.e.\ the character is multiplied by $\chi_{E/E+}^2$).
	
	We apply \cite[Thm.~3.2.1]{JoNew} again to interpolate this to a map $\eta : \cV^+ \to \cE^d_{\Sigma^+,\h^+}(K_1(\fn^+\fd^2))$. There is a natural map of Hecke algebras from level $K_1(\n^+\fd^2)$ to level $K_1(\n^+)$, where $T_v$ is mapped to $\chi_{E/E^+}(v)T_v$ for $v\nmid \n^+\fd^2$, the $U_{\pri}$ operators are sent to $\chi_{E/E^+}(\pri)U_{\pri}$, and the diamond operator $\langle u\rangle$ for $u \newmod{\n^+\fd^2}$ is sent to $\chi_{E/E^+}^2(u)\langle u\newmod{\n}\rangle,$ and one checks that this gives $\eta$. The image of $\eta$ gives a family of Hecke eigensystems through $f^+\otimes\chi_{E/E^+}$. But this must be $\sV^+$, the unique such family.
\end{proof}

We deduce that there exists a $p$-adic $L$-function $L_p(\cV^+\otimes\chi_{E/E^+})$ which, after restricting to $E^{+,\cyc}$, we consider to be in $\cD(\Gal_p^{+,\cyc},\cO(\cV^+\otimes\chi_{E/E^+}))$. Using \'etaleness over $\Sigma^+$, we see this as valued in $\cO(\Sigma^+)$. In particular,  $L_p(\cV^+)L_p(\cV^+\otimes\chi_{E/E^+})$ makes sense in $\cD(\Gal_p^{+,\cyc},\cO(\Sigma^+))$.

\subsection{The theorem in families}
We consider the line through $\lambda$, $\cW_\lambda$, obtained from $\lambda$ by allowing translation by parallel weights; that is, $\lambda' \in \cW_\lambda$ if and only if $\lambda' - \lambda$ is parallel. As $\lambda \in \cW$ is pure, we have $\cW_{\lambda} \subset \cW$. In particular, above $\Sigma \subset \cW_\lambda$ a sufficiently small neighbourhood of $\lambda$, we have a family $\cV^+$ through $f^+$ in the Hilbert eigenvariety, and a base-change family $\cV$ through $f$ in the CM eigenvariety. Then we have $p$-adic $L$-functions $L_p(\cV^+)$ and $L_p(\cV^+\otimes\chi_{E/E^+})$ over $\cO(\Sigma^+)$.

Under Conjecture \ref{conj:dim}, we have a $p$-adic $L$-function $L_p(\cV)$ over $\cV$. Let $L_p^\cyc(\cV)$ be its restriction to the cyclotomic line. Since $\cV$ is \'etale over $\Sigma$, we may consider $L_p^\cyc(V)$ to be valued in $\cO(\Sigma)$, which we identify with $\cO(\Sigma^+)$ in the natural way.

\begin{proposition}\label{prop:af families}
	Up to renormalising $L_p^\cyc(\cV)$ by an element of $\cO(\Sigma^+)^\times$, we have
	\[
	L_p^\cyc(\cV) = L_p(\cV^+)\cdot L_p(\cV^+\otimes\chi_{E/E^+})\in \cD(\Gal_p^{+,\cyc},\cO(\Sigma^+)).
	\]
\end{proposition}

\begin{proof}
	The space $\Sigma^+$ contains a Zariski-dense set of classical weights, and at a Zariski-dense subset of these, the slope of $\cV^+$ will satisfy the conditions of Proposition \ref{prop:extremely}. At any classical point $y^+ \in \cV^+$ lying above a weight in this Zariski-dense set, the interpolation properties of the $p$-adic $L$-functions -- and Proposition \ref{prop:extremely} -- imply the existence of a $p$-adic period $C_y = c_{y^+} c_{y^+\otimes\chi_{E/E^+}}/c_y$ such that
	\[
	\mathrm{sp}_{y^+}\big[L_p^\cyc(\cV)\big] = C_y \cdot\mathrm{sp}_{y^+}\big[L_p(\cV^+)\cdot L_p(\cV^+ \otimes\chi_{E/E^+})\big].
	\]
	Exactly the same proof as \cite[Prop.~6.2]{BW18} then shows that the function $y \mapsto C_y$ interpolates to an element of $\cO(\Sigma^+)$, which is equal to 1 at $f$ and which we may take to be a unit after possibly shrinking $\Sigma^+$ further. The result follows after renormalising by this element.
\end{proof}

Theorem \ref{thm:artin formalism} then follows by specialising Proposition \ref{prop:af families} at $x_{f^+}$. Note that the $p$-adic period at $f$ may be taken to be equal to 1, so we get equality without an extra $p$-adic period.

\footnotesize
\renewcommand{\refname}{\normalsize References} 
\bibliography{master_references}{}
\bibliographystyle{alpha}

\Addressesshort

\end{document}